\documentclass[12pt]{amsart}
\usepackage{amsfonts,amssymb,latexsym,amsmath, amsxtra,verbatim}
\usepackage[all]{xy}
\usepackage{graphicx}

\usepackage[OT2,T1]{fontenc}
\DeclareSymbolFont{cyrletters}{OT2}{wncyr}{m}{n}
\DeclareMathSymbol{\Sha}{\mathalpha}{cyrletters}{"58}

\theoremstyle{plain}
\newtheorem{theorem}{Theorem}[section]
\newtheorem{corollary}[theorem]{Corollary}

\newtheorem{lemma}[theorem]{Lemma}
\newtheorem{proposition}[theorem]{Proposition}

\theoremstyle{definition}

\newtheorem{definition}[theorem]{Definition}
\theoremstyle{remark}
\newtheorem*{remark}{Remark}

\numberwithin{equation}{section}

\newcommand{\R}{\mathbb R}
\newcommand{\N}{\mathbb N}
\newcommand{\Z}{\mathbb Z}

\newcommand{\cS}{\mathcal{S}}

\newcommand{\Q}{{\mathbb Q}}

\newcommand{\F}{{\mathbb F}}

\def\P{ {\bf P}}

\def\({\left(}
\def\){\right)}
\def\<{\left<}
\def\>{\right>}

\newcommand{\wt}[1]{\widetilde{#1}}

\newcommand{\abs}[1]{\left|#1\right|}

\def\GL{\text{GL}}

\def\o{{\rm o}}

\def\Arc{\text{\bf Arc}}
\def\First{\text{\bf First}}
\def\Last{\text{\bf Last}}
\def\Pars{\Pi}

\def\uP{\underline{P}}
\def\bF{\text{\bf F}}
\def\bL{\text{\bf L}}
\def\bA{\text{\bf A}}
\def\bC{\text{\bf C}}

\def\bB{{\bf B}}

\begin{document}

\title[set partition statistics: dimension and intertwining exponent ]
{
Closed expressions for averages of set partition statistics
}
\author{Bobbie Chern}
\address{Stanford University, Department of Electrical Engineering, Stanford, CA 94305}
\email{bgchern@stanford.edu}
\author{Persi Diaconis}
\address{Stanford University, Department of Mathematics and Statistics, Sequoia Hall, 390 Serra Mall, Stanford, CA 94305-4065, USA}
\email{diaconis@math.stanford.edu}
\author{Daniel M. Kane}
\address{Stanford University, Department of Mathematics, Bldg 380, Stanford, CA 94305}
\email{dankane@math.stanford.edu}
\author{Robert C. Rhoades}
\address{Stanford University, Department of Mathematics, Bldg 380, Stanford, CA 94305}
\email{rhoades@math.stanford.edu}

\thanks{}

\date{\today}
\thispagestyle{empty} \vspace{.5cm}
\begin{abstract}
In studying the enumerative theory of super characters of the group of upper triangular matrices
over a finite field we found that the moments (mean, variance and higher moments) of novel
statistics on set partitions of $[n] = \{1, 2, \cdots, n\}$
have simple closed expressions as linear combinations of shifted
bell numbers. It is shown here that
families of other statistics have similar moments. The coefficients in the linear combinations
are polynomials in $n$.  This allows exact enumeration of the moments for small
$n$ to determine exact formulae for all $n$.
\end{abstract}


\maketitle

\section{Introduction}
The set partitions of $[n]=\{1, 2, \cdots, n\}$ (denoted $\Pars(n)$) are a classical object of combinatorics.  In studying
the character theory of upper-triangular matrices (see Section \ref{sec:SuperCharacter} for
background) we
were led to some unusual statistics on set partitions.  For a set partition $\lambda$ of $n$, consider
the dimension exponent
$$d(\lambda):=  \sum_{i= 1}^\ell ( M_i - m_i + 1) - n$$
where $\lambda$ has $\ell$ blocks, $M_i$ and $m_i$ are the largest and smallest elements
of the $i$th block. How does $d(\lambda)$ vary with $\lambda$?  As shown below,
its mean and second moment are determined
in terms of the Bell numbers $B_n$
\begin{align*}
\sum_{\lambda \in \Pars(n)} d(\lambda) =& - 2B_{n+2} + (n+4) B_{n+1} \\
\sum_{\lambda \in \Pars(n)} d^2(\lambda) =& 4 B_{n+4} - (4n+15) B_{n+3} + (n^2 + 8n + 4) B_{n+2}
- (4n+3) B_{n+1} + n B_n.
\end{align*}
The right hand sides of these formulae are linear combinations of Bell numbers
with polynomial coefficients.  Dividing by $B_n$ and using asymptotitcs for
Bell numbers (see Section \ref{sec:Asymptotic}) in terms of $\alpha_n$, the positive real solution of
$ue^u = n+1$ (so $\alpha_n = \log(n) - \log \log(n) + \cdots$) gives
\begin{align*}
E(d(\lambda)) = & \( \frac{\alpha_n - 2}{\alpha_n^2}\) n^2 + O\( \frac{n}{\alpha_n}\) \\
{\rm VAR}(d(\lambda)) =& \( \frac{\alpha_n^2 - 7\alpha_n + 17}{\alpha_n^3 (\alpha_n+1)}\)^2 n^3 +
O\( \frac{n^2}{\alpha_n}\).
\end{align*}
This paper gives a large family of statistics that admit similar formulae for all moments.
These include classical statistics such as the number of blocks and number of blocks of size $i$.
It also includes many novel statistics such as $d(\lambda)$ and $c_k(\lambda)$,  the number of
$k$-crossings.  The number of $2$-crossings appears as the intertwining exponent of super
characters.

Careful definitions and statements of our main results are in Section \ref{sec:Results}.
Section \ref{sec:SuperCharacter} reviews the enumerative and probabilistic theory of set partitions,
finite groups and super-characters.  Section \ref{sec:Computational} gives computational
results; determining the coefficients in shifted Bell expressions involves summing over all set
partitions for small $n$.  For some statistics, a fast new algorithm speeds things up.  Proofs
of the main theorems are in Sections \ref{sec:Proofs2} and \ref{sec:Proofs1}.
Section \ref{sec:examples} gives a collection of examples-- moments of order up to six
for $d(\lambda)$ and further numerical data.  In a companion paper \cite{CDKR2},
the asymptotic limiting normality of $d(\lambda)$, $c_2(\lambda)$, and some other statistics
is shown.

\section{Statement of the main results}\label{sec:Results}
Let $\Pars(n)$ be the set partitions of $[n] = \{ 1, 2, \cdots, n\}$ (so $\abs{\Pars(n)} = B_n$, the $n$th
Bell number).  A variety of codings are described in Section \ref{sec:SuperCharacter}.
In this section $\lambda \in \Pars(n)$ is described as $\lambda = \bB_1 | \bB_2 | \cdots| \bB_\ell$
with $\bB_i \cap \bB_j = \emptyset$, $\cup_{i=1}^\ell  \bB_i = [n]$.  Write $
i\sim_\lambda j$ if $i$ and $j$ are in the same block of $\lambda$.  It is notationally
convenient to think of each block as being ordered.  Let $\First(\lambda)$ be the set of elements of
$[n]$ which appear first in their block and $\Last(\lambda)$ be the set of elements of $[n]$ which
occur last in their block.  Finally, let $\Arc(\lambda)$ be the set of distinct pairs of integers $(i,j)$
which occur in the same block of $\lambda$ such that $j$ is the smallest element of the block greater
than $i$. As usual, $\lambda$ may be pictured as a graph with vertex set $[n]$ and
edge set $\Arc(\lambda)$.

For example, the partition $\lambda = 1356|27|4$, represented in Figure \ref{fig:expart},
has $\First(\lambda) = \{ 1, 2, 4\}$, $\Last(\lambda) = \{ 6, 7, 4\}$, and $\Arc(\lambda) = \{(1,3), (3,5), (5,6), (2,7)\}.$
\begin{figure}[!h]
\center
\includegraphics[scale=0.5]{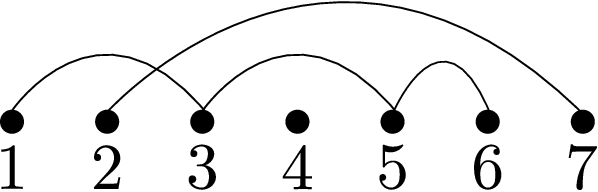}
\caption{An example partition $\lambda = 1356|27|4$}
\label{fig:expart}
\end{figure}

A \emph{statistic} on $\lambda$ is defined by counting the number of occurrences of
\emph{patterns}.  This requires some notation.
\begin{definition} \ \
\begin{enumerate}
\item[($i$)]
A \emph{pattern $\uP$ of length $k$} is defined by a set partition $P$ of $[k]$
and subsets $\bF(\uP), \bL(\uP)\subset[k]$ and $ \bA(\uP), \bC(\uP)  \subset [k]\times [k]$.
Let $\uP = (P, \bF, \bL, \bA, \bC)$.

\item[$(ii)$]
An \emph{occurrence} of a pattern $\uP$ of length $k$ in $\lambda \in \Pars(n)$
is $s = (x_1, \cdots, x_k)$ with $x_i \in [n]$ such that
\begin{enumerate}
\item[(1)] $x_1< x_2< \cdots < x_k$.
\item[(2)] $x_i \sim_\lambda x_j$ if and only if $i \sim_P j$.
\item[(3)] $x_i\in \First(\lambda)$ if $i \in \bF(\uP)$.
\item[(4)] $x_i \in \Last(\lambda)$ if $i \in \bL(\uP)$.
\item[(5)] $(x_i,x_j)\in \Arc(\lambda)$ if $(i,j)\in \bA(\uP)$.
\item[(6)] $\abs{x_i - x_j} = 1$ if $(i,j) \in \bC(\uP)$.
\end{enumerate}
Write $s \in_{\uP} \lambda$ if $s$ is an occurrence of $\uP$ in $\lambda$.

\item[$(iii)$] A \emph{simple statistic} is defined by a pattern $\uP$ of length $k$
and $Q \in \Z[y_1, \cdots, y_k, m]$.  If $\lambda \in \Pars(n)$ and $s = (x_1, \cdots, x_k)
\in_{\uP} \lambda$,
write
$Q(s) = Q\mid_{y_i = x_i, m = n}$. Let
$$f(\lambda) = f_{\uP, Q} (\lambda) := \sum_{s\in_{\uP} \lambda} Q(s).$$
Let the \emph{degree} of a simple statistic $f_{P,Q}$ be the sum of the length of $\uP$
and the degree of $Q$.

\item[$(iv)$]
A \emph{statistic} is a finite $\Q$-linear combination of simple statistics.
The degree of a statistic is defined to be the minimum over such representations
 of the maximum degree of any appearing simple statistic.
\end{enumerate}
\end{definition}

\begin{remark}
In the notation above, $\bF(\uP)$ is the set of firsts elements, $\bL(\uP)$ is the set of lasts,
 $\bA$ is the arc set of the pattern,  and $\bC(\uP)$ is the set of consecutive elements.
\end{remark}

\noindent
{\bf Examples.}
\begin{enumerate}
\item \underline{Number of Blocks in $\lambda$}:
$$\ell(\lambda) = \sum_{ \begin{subarray}{c} 1\le x \le n \\ x \text{ is smallest element in its block}\end{subarray} } 1.$$
Here $\uP$ is a pattern of length 1, $\bF(\uP) = \{ 1\}$,
$\bL(\uP)  = \bA(\uP) = \bC(\uP) = \emptyset$ and $Q(y,m) = 1$.
Similarly, the $n$th moment of $\ell(\lambda)$ can be computed using
\begin{equation*}\label{eqn:moments_ell}
\binom{\ell(\lambda)}{k}  =
f_{\uP_k, 1}(\lambda)
\end{equation*}
where $\uP_k $ is the pattern of length $k$ corresponding to $P$, the partitions of $[k]$ into
blocks of size $1$, with $\bF(\uP_k) = \{ 1, 2, \cdots, k\}$, and $\bL(\uP_k) = \bA(\uP_k) = \bC(\uP_k) = \emptyset$.

\item \underline{Number of blocks of size $i$}:
Define a pattern  $\uP_i$ of length $i$  by: (1)  all elements of $[i]$ are equivalent,
(2) $\bF(\uP_i) = \{ 1\}$,
(3) $\bL(\uP_i) = \{ i\}$,
(4) $\bA(\uP_i) = \{ (1, 2), \cdots, (i-1, i)\}$
and (5) $\bC(\uP_i) = \emptyset$.
Then
\begin{equation}\label{eqn:Xi_def}
X_i(\lambda) := f_{\uP_i, 1}(\lambda)
\end{equation}
 is the number of $i$-blocks in $\lambda$.
(If $i = 1$, $\bA(\uP_1) = \emptyset$.) Similarly, the moments of the number of blocks
of size $i$ is a statistic.  See Theorem \ref{thm:algebraAbstract}.

\item \underline{$k$-crossings}:
A $k$-crossing \cite{cddsy} of a $\lambda \in \Pars(n)$ is a sequence
of arcs $(i_t, j_t)_{1\le t \le k} \in \Arc(\lambda)$ with
$$i_1 < i_2  < \cdots < i_k < j_1 < j_2 < \cdots < j_k.$$
The statistic $cr_k(\lambda)$ which counts the number of $k$-crossings of $\lambda$
can be represented by a pattern $\uP = (P, \bF, \bL, \bA, \bC)$ of length $2k$ with
(1)  $i \sim_P k+i$ for $i = 1, \cdots, k$,
(2) $\bF(\uP) = \bL(\uP) = \emptyset$,
(3)  $\bA(\uP) = \{ (1, k+1), (2, k+2), \cdots, (k, 2k)\}$,
and (4) $\bC(\uP) = \emptyset$.

Partitions with $cr_2(\lambda) = 0$ are in bijection with Dyck paths and so
are counted by the Catalan numbers $C_n = \frac{1}{n+1} \binom{2n}{n}$
(see Stanley's second volume on enumerative combinatorics \cite{stanley2}).
Partitions without crossings have proved themselves to be very interesting.
Crossing seems to have been introduced by Krewaras \cite{kreweras}.  See Simion's \cite{simion} for an
extensive survey and
Chen, Deng, Du, and Stanley   \cite{cddsy} and Marberg \cite{marberg2} for more recent appearances of
this statistic.
The statistic $cr_2(\lambda)$ appears as the intersection exponent in Section \ref{sec:GroupTheory}
below.

\item \underline{Dimension Exponent}: The dimension exponent described in the introduction
is a linear combination of the number of blocks (a simple statistic of degree 1),
the last elements of the blocks (a simple statistic of degree 2),
and the first elements of the blocks (a simple statistic of degree 2).
Precisely, define $f_{firsts}(\lambda) := f_{\uP, Q}(\lambda)$ where $\uP$ is the pattern of length 1, with
$\bF(\uP) = \{ 1\}$, $\bL(\uP) = \bA(\uP)= \bC(\uP)  = \emptyset$ and $Q(y, m) = y$. Similarly,
let $f_{lasts}(\lambda) := f_{\uP, Q}(\lambda)$ where $\uP$ is the pattern of length 1, with
$\bL(\uP) = \{ 1\}$, $\bF(\uP) = \bA(\uP)= \bC(\uP) = \emptyset$ and $Q(y, m) = y$.
Then $$d(\lambda) = f_{lasts}(\lambda) - f_{firsts}(\lambda) +\ell(\lambda) - n.$$

\item \underline{Levels}: The number of levels in $\lambda$ , denoted $f_{levels}(\lambda)$,
(see page 383 of \cite{mansour} or Shattuck \cite{shattuck}) is the number of $i$ such that
$i$ and $i+1$ appear in the same block of $\lambda$.  We have
$$f_{levels}(\lambda) = f_{\uP, Q}(\lambda)$$
where $\uP$ is a pattern of length 2 with $\bC(\uP) = \bA(\uP) = \{ (1, 2)\}$ and $\bA(\uP) = \bF(\uP) = \emptyset$.

\item The maximum block size of a partition is \emph{not} a statistic in this notation.
\end{enumerate}

The set of all statistics on $\cup_{n=0}^\infty \Pars(n) \rightarrow \Q$ is a filtered algebra.
\begin{theorem}\label{thm:algebraAbstract}
Let $\cS $ be the set of all set partition statistics thought of as functions $f:\bigcup_n \Pars(n)\rightarrow \Q$. Then $\cS$ is closed under the operations of pointwise scaling, addition and multiplication. In particular, if $f_1,f_2\in \cS$ and $a\in\Q$, then there exist partition statistics $g_a,g_+,g_*$ so that for all set partitions $\lambda$,
\begin{eqnarray*}
a f_1(\lambda) & = & g_a(\lambda)\\
f_1(\lambda) + f_2(\lambda) & = & g_+(\lambda)\\
f_1(\lambda)\cdot f_2(\lambda) & = & g_*(\lambda).
\end{eqnarray*}
Furthermore, $\deg(g_a)\leq\deg(f_1)$, $\deg(g_+)\leq\max(\deg(f_1),\deg(f_2))$, and $\deg(g_*)\leq \deg(f_1)+\deg(f_2)$. In particular, $\cS$ is a filtered $\Q$-algebra under these operations.
\end{theorem}
\begin{remark}
Properties of this algebra remain to be discovered.
\end{remark}

\begin{definition}
A \emph{shifted Bell polynomial} is any function $R: \N \to \Q$ given by
$$R(n) = \sum_{I \le j \le K} Q_{j}(n) B_{n+j}$$
where $I, K\in \Z$ and each $Q_j(x) \in \Q[x]$.
i.e. it is  a finite sum of polynomials multiplied by shifted Bell numbers.
Call $K$ the \emph{upper shift degree} of $R$ and $I$ the \emph{lower shift degree} of $R$.
\end{definition}

Our first main theorem shows that the aggregate of a statistic is a shifted Bell polynomial.

\begin{theorem}\label{thm:structureAbstract}
For any statistic, $f$ of degree $N$, there
exists a shifted Bell polynomial $R$ such that for all $n\ge 1$
$$
M(f;n):= \sum_{\lambda \in \Pars(n)} f(\lambda) = R(n).$$
 Moreover,
\begin{enumerate}
\item  the upper shift index of $R$ is at most $N$
and the lower shift index is bounded below by  $-k$, where $k$ is the size of the pattern associated $f$.
\item  the degree of the polynomial coefficient of $B_{n+N-j}$ in $R$ is bounded by $j$ for $j \le N$ and by
$j-1$ for $j > N$.
\end{enumerate}
\end{theorem}

The following collects the shifted Bell polynomials for the aggregates of the statistics given above.
\noindent
{\bf Examples.}
\begin{enumerate}
\item \underline{Number of blocks in $\lambda$}:
$$M(\ell; n) = B_{n+1}  - B_n.$$
This is elementary and is established in Proposition \ref{prop:baby} below.
\item \underline{Number of blocks of size $i$}:
$$M(X_i; n) = \binom{n}{i} B_{n-i}.$$
This is also elementary and is established in Proposition \ref{prop:baby} below.
\item \underline{$2$-crossings}: Kasraoui \cite{kasraouiZ} established
$$M(cr_2; n) = \frac{1}{4} \( - 5 B_{n+2} + (2n+9) B_{n+1} + (2n+1) B_{n}  \).$$
\item \underline{Dimension Exponent}:
$$M(d; n) = -2B_{n+2} + (n+4) B_{n+1}.$$
This is given in Theorem \ref{thm:structure} below.
\item \underline{Levels}:
Shattuck \cite{shattuck} showed that
$$M(f_{levels}; n) = \frac{1}{2}(B_{n+1} - B_n - B_{n-1}).$$
It is amusing that this implies that $B_{3n} \equiv B_{3n+1} \equiv 1\pmod{2}$ and $B_{3n+2} \equiv 0 \pmod{2}$ for all $n\ge 0$.
\end{enumerate}

\begin{remark}
Chapter 8 of Mansour's book \cite{mansour} and the research papers
 \cite{kasraouiZ, ms, kmw} contain many other examples of
statistics which have shifted Bell polynomial aggregates.
We believe that each of these statistics is covered by our class of statistics.
\end{remark}

\section{Set Partitions, Enumerative Group Theory and Super-characters}\label{sec:SuperCharacter}
This section presents background and a literature review of set partitions, probabilistic
and enumerative group theory and super-character theory for the upper triangular group over a finite
field.  Some sharpenings of our general theory are given.

\subsection{Set Partitions}\label{sec:SetPartitions}
Let $\Pars(n,k)$ denote the set partitions of $n$ labelled objects with $k$ blocks and $\Pars(n) = \cup_{k} \Pars(n,k)$; so $\abs{\Pars(n,k)} = S(n,k)$ the Stirling number of the second kind and $\abs{\Pars(n)} = B_n$
the $n$th Bell number.  The enumerative theory and applications of these basic objects is developed in
Graham-Knuth-Patashnick \cite{gkp}, Knuth \cite{knuth4a}, Mansour \cite{mansour}
and Stanley \cite{stanley1}.
There are many familiar equivalent codings
\begin{itemize}
\item Equivalence relations on $n$ objects with $k$ blocks
$$1|2|3\ , \hspace{.2in} 12|3\ , \hspace{.2in} 13|2\ , \hspace{.2in} 1|23\ , \hspace{.2in} 123$$
\item Binary, strictly upper-triangular zero-one matrices with no
two ones in the same row or column.
(Equivalently, rook placements on a triangular Ferris board (Riordan \cite{riordan})
$$ \( \begin{array}{ccc} 0&0&0\\0&0&0\\0&0&0\end{array}\), \ \
\( \begin{array}{ccc} 0&1&0\\0&0&0\\0&0&0\end{array}\), \ \
\( \begin{array}{ccc} 0&0&1\\0&0&0\\0&0&0\end{array}\),  \ \
\( \begin{array}{ccc} 0&0&0\\0&0&1\\0&0&0\end{array}\),  \ \
\( \begin{array}{ccc} 0&1&0\\0&0&1\\0&0& 0\end{array}\)
$$
\item Arcs on $n$ points
\begin{figure}[!h]
\center
\includegraphics[scale=0.4]{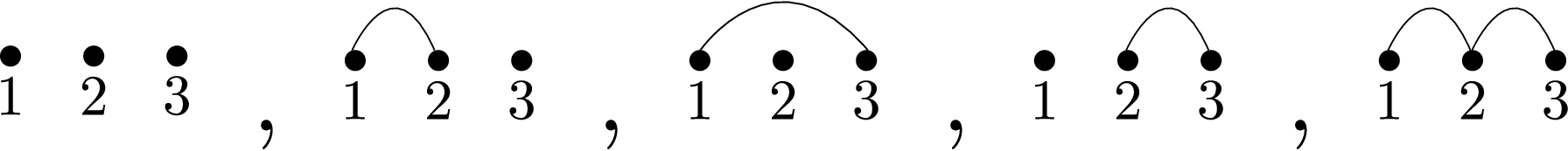}
\end{figure}
\item Restricted growth sequences $a_1, a_2, \ldots, a_n$; $a_1=0, a_{j+1} \le
1 + \max(a_1, \ldots, a_j)$  for $1 \leq j < n$
(Knuth \cite{knuth4a}, p. 416)
$$012 \ , \hspace{.2in} 001\ , \hspace{.2in} 010\ , \hspace{.2in} 011\ , \hspace{.2in} 000$$
\item Semi-labelled trees on $n+1$ vertices
\begin{figure}[!h]
\center
\includegraphics[scale=0.4]{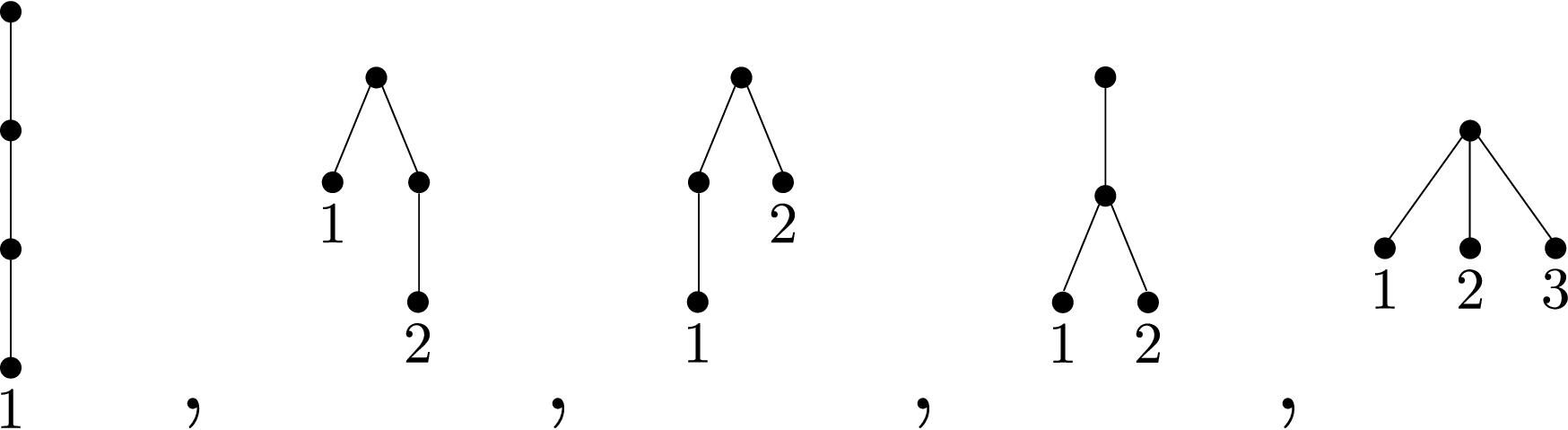}
\end{figure}
\item Vacillating Tableau: A sequence of partitions $\lambda^0, \lambda^1, \cdots, \lambda^{2n}$
with $\lambda^0 = \lambda^{2n} = \emptyset$ and $\lambda^{2i+1}$ is obtained from $\lambda^{2i}$
by doing nothing or deleting a square and $\lambda^{2i}$ is obtained from $\lambda^{2i-1}$ by doing nothing
or adding a square (see \cite{cddsy}).
\end{itemize}

The enumerative theory of set partitions begins with Bell polynomials. Let
$B_{n,k}(w_1, \cdots, w_n) = \sum_{\lambda \in \Pars(n,k)} \prod w_i^{X_i(\lambda)}$
with $X_i(\lambda)$ the number of blocks in $\lambda$ of size $i$; so set
$B_n(w_1, \cdots, w_n) = \sum_{k} B_{n,k}(w_1, \cdots, w_n)$ and
$B(t) = \sum_{n=0}^\infty B_n(\underline{w}) \frac{t^n}{n!}.$
A classical version of the exponential formula gives
\begin{equation}\label{eqn:setpartitions_gf}
B(t) = e^{\sum_{n=1}^\infty w_n \frac{t^n}{n!}}.
\end{equation}
These elegant formulae have been used by physicists and chemists to understand
fragmentation processes (\cite{pitman2} for extensive references).  They also underlie the theory
of polynomials of binomial type \cite{garsia, knuthBiType}, that is, families
$P_n(x)$ of polynomials satisfying
$$P_n(x+y) = \sum P_k(x) P_{n-k}(y).$$
These unify many combinatorial identities, going back to Faa de Bruno's formula for the Taylor
series of the composition of two power series.

There is a healthy algebraic theory of set partitions. The partition algebra of
\cite{hr} is based on a natural product on $\Pars(n)$ which first arose in diagonalizing
the transfer matrix for the Potts model of statistical physics.
The set of all set partitions $\bigcup_n \Pars(n)$ has a Hopf algebra structure which is a
general object of study in \cite{am}.

Crossings and nestings of set partitions is a emerging topic, see \cite{ cddsy, kz, kasragui} and their
references. Given $\lambda \in \Pars(n)$ two arcs
$(i_1, j_1)$ and $(i_2, j_2)$ are said to \emph{cross} if $i_1 < i_2 < j_1 < j_2$ and \emph{nest}
if $i_1 < i_2  < j_2 < j_1$.  Let $cr(\lambda)$ and $ne(\lambda)$ be the number of crossings and nestings.
One striking result: the crossings and nestings are equi-distributed (\cite{kz} Corollary 1.5), they show
$$\sum_{\lambda \in \Pars(n)} x^{cr(\lambda)} y^{ne(\lambda)} = \sum_{\lambda \in \Pars(n)} x^{ne(\lambda)}
y^{cr(\lambda)}.$$
As explained in Section \ref{sec:GroupTheory} below, crossings arise in a group theoretic context and are covered
by our main theorem.
Nestings are also a statistic.
This crossing and nesting literature develops a parallel theory for crossings and nestings of perfect
matchings (set partitions with all blocks of size 2).  Preliminary works suggest that our main theorem carry over to matchings with $B_n$ reduced to $(2n)!/2^nn!$.

Turn next to the probabilistic side: What does a `typical' set partition `look like'? For example, under
the uniform distribution on $\Pars(n)$
\begin{itemize}
\item What is the expected number of blocks?
\item How many singletons (or  blocks of size $i$) are there?
\item What is the size of the largest block?
\end{itemize}
The Bell polynomials can be used to get moments.  For example:
\begin{proposition}\label{prop:baby}
\
\begin{enumerate}
\item[$(i)$] Let $\ell(\lambda)$ be the number of blocks.  Then
\begin{align*}
m(\ell;n) :=& \sum_{\lambda \in \Pars(n)} \ell(\lambda) = B_{n+1} - B_n \\
m(\ell^2;n)  =& B_{n+2} - 3 B_{n+1}  + B_n \\
m(\ell^3;n) =& B_{n+3} - 6B_{n+2} + 8 B_{n+1} B_{n+1} - B_n
\end{align*}
\item[$(ii)$]
Let $X_1(\lambda)$ be the number of singleton blocks, then
\begin{align*}
m(X_1;n) =& nB_{n-1} \\
m(X_1^2;n) =& nB_{n-1} + n(n-1)B_{n-2}
\end{align*}
\end{enumerate}
\end{proposition}
In accordance with our general theorem, the right hand sides of $(i), (ii)$ are
shifted Bell polynomials. To make contact with results above, there is a direct proof
of these classical formulae.
\begin{proof}
Specializing the variables in the generating function \eqref{eqn:setpartitions_gf} gives a two variable
generating functions for $\ell$:
$$\sum_{n=0}^\infty \sum_{\lambda \in \Pars(n)} y^{\ell(\lambda)} \frac{x^n}{n!} = \sum_{\begin{subarray}{c} n\ge 0 \\ \ell \ge 0 \end{subarray}} S(n,\ell) y^\ell \frac{x^n}{n!} = e^{y(e^x-1)}.$$
Differentiating with respect to $y$ and setting $y=1$ shows that $m(\ell;n)$ is the coefficient of
$\frac{x^n}{n!}$ in $(e^x-1) e^{e^x-1}$.
Noting that $$\frac{\partial}{\partial x} e^{e^x-1} = e^x e^{e^{x-1}} = \sum_{n=0}^\infty B_{n+1} \frac{x^n}{n!}$$
yields $m(\ell) = B_{n+1} - B_n$. Repeated differentiation gives the higher moments.

For $X_1$, specializing variables gives
$$\sum_{n=0}^\infty \sum_{\lambda \in \Pars(n)} y^{X_1(\lambda)} \frac{x^n}{n!}
= e^{e^x - 1 - x + yx}.$$
Differentiation with respect to $y$ and settings $y=1$ readily yields the claimed results.
\end{proof}

The moment method may be used to derive limit theorems. An easier, more systematic method is due to
Fristedt \cite{fristed}. He interprets the factorization of the generating function $B(t)$ in \eqref{eqn:setpartitions_gf}
as a conditional independence result and uses ``dePoissonization'' to get results for finite $n$.
Let $X_i(\lambda)$ be the number of blocks of size $i$.  Roughly, his results say that
$\{X_i\}_{i=1}^n$ are asymptotically independent and of size $(\log(n))^i/i!$.
More precisely, let $\alpha_n$ satisfy $\alpha_n e^{\alpha_n} = n+1$ (so $\alpha_n = \log(n) -
\log\log(n) + o(1)$).  Let $\beta_i = \alpha_n^i/i!$ then
$$\P\big\{ \frac{X_i - \beta_i}{\sqrt{\beta_i}} \le x\big\} = \Phi(x) + o(1)$$
where $\Phi(x) = \frac{1}{\sqrt{2\pi}} \int_{-\infty}^x e^{-u^2/2} du.$ Fristdt also has a
description of the joint distribution of the largest blocks.
\begin{remark}
It is typical to expand the asymptotics in terms of $u_n$ where $u_n e^{u_n} = n$. In this notation
$u_n$ and $\alpha_n$ differ by $O(1/n)$.
\end{remark}

The number of blocks $\ell(\lambda)$ is asymptotically normal when standardized by its
mean $\mu_n \sim \frac{n}{\log(n)}$ and variance $\sigma_n^2 \sim \frac{n}{\log^2(n)}$.  These
are precisely given by Proposition \ref{prop:baby} above.  Refining this, Hwang \cite{hwang} shows
$$\P\big\{ \frac{\ell - \mu_n}{\sigma_n} \le x\big\} = \Phi(x) + O\( \frac{\log(n)}{\sqrt{n}}\).$$

Stam \cite{stam} has introduced a clever algorithm for random uniform sampling of set partitions
in $\Pars(n)$.  He uses this to show that if $W(i)$ is the size of the block containing $i$, $1\le i \le k$, then
for $k$ finite and $n$ large $W(i)$ are asymptotically independent and normal with mean
and variance asymptotic to $\alpha_n$.  In \cite{CDKR2} we use Stam's algorithm to prove
the asymptotic normality of $d(\lambda)$ and $cr_2(\lambda)$.

Any of the codings above lead to distribution questions.  The upper-triangular representation
leads to the study of the dimension and crossing statistics,
the arc representation suggests crossings,
nestings and even the number of
arcs, i.e. $n - \ell(\lambda)$.
Restricted growth sequences suggest the number of zeros, the number of leading zeros,
largest entry.  See Mansour \cite{mansour} for this and much more.
Semi-labelled trees suggest the number of leaves, length of the longest path
from root to leaf and
various measures of tree shape (eg. max degree).  Further probabilistic aspects of uniform
set partitions can be found in  \cite{pitman1, pitman2}.

\subsection{Probabilistic Group Theory}\label{sec:probabilisticGT}
One way to study a finite group $G$ is to ask what `typical' elements `look
like'.  This program was actively begun by Erd\"os and Turan \cite{et1, et2,et3, et4, et5, et6, et7} who
focused on the symmetric group $S_n$.  Pick a permutation $\sigma$ of $n$
at random and ask the following:
\begin{itemize}
\item How many cycles in $\sigma$? (about $\log n$)
\item What is the length of the longest cycle? (about $0.61n$)
\item How many fixed points in $\sigma$? (about 1)
\item What is the order of $\sigma$? (roughly $e^{(\log n)^2/2}$)
\end{itemize}
In these and many other cases the questions are answered with much more
precise limit theorems.  A variety of other classes of groups have been
studied.  For finite groups of Lie type
see \cite{fulman} for a survey and \cite{fg}
for wide-ranging applications.  For $p$-groups
see \cite{newman}.

One can also ask questions about `typical' representations.  For example, fix a
conjugacy class $C$ (e.g. transpositions in the symmetric group),
what is the distribution of $\chi_\rho(C)$ as $\rho$ ranges over
irreducible representations \cite{fulman, kerov, turan}.  Here, two
probability distributions are natural, the uniform distribution on
$\rho$ and the Plancherel measure ($\Pr(\rho) =
d_\rho^2/|G|$ with $d_\rho$ the dimension of $\rho$).
Indeed, the behavior of the `shape' of a random partition of $n$ under the
Plancherel measure for $S_n$ is one of the most celebrated results in modern
combinatorics.  See Stanley's  \cite{stanley} for a survey with references to the work of
Kerov-Vershik \cite{kv}, Logan-Shepp \cite{ls}, Baik-Deift-Johansson \cite{bdj} and many others.

The above discussion focuses on finite groups.  The questions make sense for
compact groups.  For example, pick a random matrix from Haar measure on the unitary group
$U_n$ and ask: What is the distribution of its eigenvalues? This leads to the very
active subject of random matrix theory.  We point to the wonderful monographs
of Anderson-Guionnet-Zietouni \cite{agz}
and Forrester \cite{forrester} which have extensive surveys.


\subsection{Super-character theory}\label{sec:GroupTheory}
Let $G_n(q)$ be the group of $n\times n$ matrices which are upper triangular with ones on the diagonal.
The group $G_n(q)$ is the Sylow $p$-subgroup of $\GL_n(\F_q)$ for $q = p^a$.  Describing the irreducible characters of $G_n(q)$ is a well-known
wild problem.  However, certain unions of conjugacy classes, called superclasses, and certain characters, called supercharacters,
have an elegant theory.  In fact, the theory is rich enough to
provide enough understanding of the Fourier analysis on the group
to solve certain problems, see  the work of Arias-Castro, Diaconis, and Stanley \cite{ADS}.
These superclasses and supercharacters were developed by Carlos Andr\'e \cite{andre1, andre2, andre3}
and Ning Yan \cite{yan}. Supercharacter theory is a growing subject.
See \cite{aguiar, bergeron,  di, dt, marberg, marberg2} and their references.

For the groups $G_n(q)$ the supercharacters are determined by a set partition of $[n]$
and a map
from the set partition to the group $\F_q^\ast$.   In the analysis of these characters there are two important statistics, each of which
only depends on the set partition.  The dimension exponent is denoted $d(\lambda)$ and the intertwining exponent is denoted $i(\lambda)$.

Indeed if $\chi_\lambda$ and $\chi_\mu$ are two supercharacters then
$$\dim\( \chi_\lambda\) = q^{d(\lambda)} \  \ \ \text{ and } \ \ \
\left< \chi_\lambda, \chi_\mu \right> = \delta_{\lambda, \mu} q^{i(\lambda)}.$$
While $d(\lambda)$ and $i(\lambda)$ were originally defined in terms of the upper triangular
representation (for example, $d(\lambda)$ is the sum of the horizontal distance from
the `ones' to the
super diagonal) their definitions can be given in terms of blocks or arcs:
\begin{equation}
d(\lambda) :=  \sum_{e \frown f \in \Arc(\lambda)} \( f - e - 1\)
\end{equation}
and
\begin{equation}
i(\lambda) := \sum_{\begin{subarray}{c}
e_1 < e_2 < f_1 < f_2 \\
e_1 \frown f_1 \in \Arc(\lambda) \\
e_2 \frown f_2  \in \Arc(\lambda)
\end{subarray} } 1
\end{equation}
\begin{remark}
Notice that $i(\lambda) = cr_2(\lambda)$ is the number of 2-crossings which were introduced in the
previous sections.
\end{remark}

Our main theorem shows that there are explicit formulae for every moment of these statistics.
The following represents a sharpening using special properties of the dimension exponent.
\begin{theorem}\label{thm:structure}
For each $k \in \{0, 1, 2, \cdots\}$ there exists a closed form expression
$$M(d^k; n):= \sum_{\lambda \in \Pars(n)} d(\lambda)^k = P_{k, 2k}(n) B_{n+2k} + P_{k, 2k-1}(n) B_{n+2k-1} + \cdots + P_{k,0}(n) B_{n}$$
where each $P_{k, 2k-j}$ is a polynomial with rational coefficients.
Moreover, the degree of $P_{k, 2k-j}$ is
$$ \begin{cases} j & j \le k \\ k - \lceil \frac{j-k}{2} \rceil & j  > k \end{cases}.$$
For example,
\begin{align*}
\sum_{\lambda \in \Pars(n)} d(\lambda) =& -2B_{n+2} + (n+4)B_{n+1}  \\
\sum_{\lambda \in \Pars(n)} d(\lambda)^2 =& 4 B_{n+4} - (4n+15) B_{n+3} + (n^2 + 8n+9) B_{n+2} - (4n+3) B_{n+1} + nB_n
\end{align*}
\end{theorem}

\begin{remark}
See Section \ref{sec:Data} for the moments with $k\le 6$ and see \cite{webpage} for the moments
with $k \le 22$.
The first moment may be deduced easily from results of Bergeron and Thiem \cite{BT}.  Note, they seem to
have an index which differs by one from ours.
\end{remark}

\begin{remark}
Theorem \ref{thm:structure} is stronger than what is obtained directly from
Theorem \ref{thm:structureAbstract}. For example, the lower shift index is $0$, while the best that can be
obtained from Theorem \ref{thm:structureAbstract} is a lower shift index of $-k$.
This theorem is proved by working directly with the generating
function for a generalized statistic on ``marked set partitions''.
These set partitions are introduced in Section \ref{sec:Computational}.
\end{remark}

Asymptotics for the Bell numbers yield the following asymptotics for the moments.
The following result gives some asymptotic information about these moments.
\begin{theorem}\label{thm:asymptotics}
Let $\alpha_n = \log(n) - \log\log(n) +o(1)$
be the positive real solution of $ue^u = n+1$.
Then $$E\( d(\lambda) \) =   \(\frac{\alpha_n - 2}{\alpha_n^2}\)
n^2 + O\( n\alpha_n^{-1}\).$$
Let $S_k(d; n):= \frac{1}{B_n} \sum_{\lambda \in \Pars(n)}  \(d(\lambda) - M(d;n)/B_n\)^k$ be the symmetrized moments of the dimension
exponent. Then
\begin{align*}
S_2(d;n)  =&  \(\frac{\alpha_n^2 - 7\alpha_n+17}{\alpha_n^3(\alpha_n+1)} \)  n^3 + O\(n^2\alpha_n^{-1}\) \\
S_3 (d;n) = &
    \(-\frac{881}{3}  -244 \alpha_n + 145 \alpha_n^{2}  -\frac{83}{3} \alpha_n^{3} + 2 \alpha_n^{4}\)\frac{n^{4}}{\alpha_n^4(\alpha_n+1)^3}  + O\( n^3\alpha_n^{-1}\)  \\
 \end{align*}
\end{theorem}
\begin{remark}
Asymptotics for $S_k(d;n)$ with  $k = 1, 2, 3, 4, 5, 6$ and  with further accuracy are in Section \ref{sec:examples}.
\end{remark}

Analogous to these results for the dimension exponent  are
the following results for the intertwining exponent.
\begin{theorem}\label{thm:structureInter}
For each $k \in \{0, 1, 2, \cdots\}$ there exists a closed form expression
$$M(i^k; n):= \sum_{\lambda \in \Pars(n)} i(\lambda)^k =
Q_{k, 2k}(n) B_{n+2k} +  \cdots + Q_{k,0}(n) B_{n}+ \cdots +
Q_{k,-k} (n) B_{n-k}$$
where each $Q_{k, 2k-j}$ is a polynomial with rational coefficients.
Moreover, the degree of $Q_{k, 2k-j}$ is
bounded by $j$.
For example,
\begin{align*}
M(i; n) =& \frac{1}{4}\((2n+1)B_n +(2n+9)B_{n+1}-5B_{n+2}\)\\
M(i^2; n) =& \frac{1}{144}\( (36n^2 +24n - 23) B_{n} + (72n^2 + 72n - 260) B_{n+1}  \right. \\ & \ \ \left. + (36n^2 + 156n + 489) B_{n+2}
- (180n +814)B_{n+3} + 225 B_{n+4} \).
\end{align*}
\end{theorem}
\begin{remark}
The expression for $M(i;n) = M(cr_2; n)$ was established first by Kasraoui (Theorem 2.3 of \cite{kasraouiZ}).
\end{remark}
\begin{remark}
Theorem \ref{thm:structureInter} is deduced directly from Theorem \ref{thm:structureAbstract}.
The shifted Bell polynomials for $M(i^k;n)$ for $k\le 5$ are given in Section \ref{sec:Data} and see
\cite{webpage} for the aggregates with $k \le 12$.
\end{remark}
\begin{remark}
Amusingly, the formula for $M(i;n)$ implies that the sequence $\{ B_n\}_{n=0}^\infty$ taken modulo 4 is
periodic of length 12 beginning with $\{1, 1, 2, 1, 3, 0, 3, 1, 0, 3, 3, 2\}$. Similarly, the formula for
$M(i^2;n)$ shows that the sequence is periodic modulo $9$ (respectively 16) with period 39 (respectively 48).
For more about such periodicity see the papers of Lunnon, Pleasants, and Stephens \cite{lps} and
Montgomery, Nahm, and Wagstaff \cite{mnw}.
\end{remark}

In analogy with Theorem \ref{thm:asymptotics} there is the following
asymptotic result.
\begin{theorem}\label{thm:asymptoticsInter}
With $\alpha_n$ as above,
$$E(i(\lambda)) = \( \frac{2\alpha_n - 5}{4\alpha_n^2} \) n^2 + O(n\alpha_n^{-1}).$$
Let $S_k( i; n) =  \frac{1}{B_n} \sum_{\lambda \in \Pars(n)} \( i(\lambda)  - M(i, n)/B_n\)^k$.
Then,
\begin{align*}
S_2(i; n) =&\frac{3 \alpha_n^2 - 22\alpha_n + 56}{9\alpha_n^3 (\alpha_n+1)} n^3  + O(n^2\alpha_n^{-1}) \\
S_3(i; n) =&   \frac{(\alpha_n - 5)(4\alpha_n^3 - 31\alpha_n^2 + 100\alpha_n + 99)}{8 \alpha_n^4 (\alpha_n+1)^3} n^4 + O(n^3\alpha_n^{-3}) \\
\end{align*}
\end{theorem}

Theorems \ref{thm:structure} and \ref{thm:structureInter} show that there will be
closed formulae for all of the moments of these statistics. Moreover, these theorems give
bounds for the number of terms in the summand and the degree of each of the polynomials.
Therefore, to compute the formulae
it is enough to compute enough values for $M(d^k;n)$ or $M(i^k;n)$ and then to do linear algebra
to solve for the coefficients of the polynomials.  For example,
$M(d;n)$  needs
$P_{1, 2}(n)$ which has degree at most $0$, $P_{1,1}(n)$ which has degree at most $1$,
and $P_{1, 0}(n)$ which has degree at most 0.  Hence, there are $4$ unknowns, and so only
$M(d;n)$ for $n = 1,2, 3, 4$ are needed to derive the formula for the expected value of the
dimension exponent.

\section{Computational Results}\label{sec:Computational}

Enumerating set partitions and calculating these statistics would take time
$O(B_n)$ (see Knuth's volume \cite{knuth4a} for discussion of how to generate all set partitions
of fixed size, the book of Wilf and Nijenhuis \cite{nw}, or the website \cite{ruski} of Ruskey).
This section introduces a recursion for computing the number of
set partitions of $n$ with a given dimension or intertwining exponent  in time $O(n^4)$.
The recursion follows by introducing a notion of ``marked'' set partitions. This generalization
seems useful in general when computing statistics which depend on the internal
structure of a set partition.
The results may then be used with Theorems \ref{thm:structure} and \ref{thm:structureInter} to
find exact formulae for the moments.  Proofs are given in Section \ref{sec:Proofs2}.

For a set partition
$\lambda$ mark each block either open or closed.  Call such a
partition a {\it marked set partition}.
For each marked set partition $\lambda$
of $[n]$ let $o(\lambda)$ be the number of open blocks
of $\lambda$ and $\ell(\lambda)$ be the total number of blocks of
$\lambda$.
(Marked set partitions may be thought of as what is obtained when
considering a set partition of a potentially larger set and restricting it to $[n]$.
The open blocks are those that will become larger upon adding more elements of this larger set, while the closed blocks are those that will not.)
With this notation define the dimension of $\lambda$
with blocks $\bB_1, \bB_2, \cdots$ by
\begin{equation}
\wt{d}(\lambda) = \left(\sum_{\begin{subarray}{c} \bB_j \\ \bB_j \text{ is closed} \end{subarray}}
\max(\bB_j)\right)  - \left( \sum_{\bB_j } \min(\bB_j) \right) + \ell(\lambda) + n\( \o(\lambda)-1\).
\end{equation}
It is clear that if $o(\lambda) = 0$, then $\lambda$ may be thought of as a usual
`unmarked' set partition and
 $\wt{d}(\lambda) = d(\lambda)$ is the dimension exponent of $\lambda$.
 Define
\begin{equation}
f(n; A, B):= \big\{\lambda \in \Pars(n):  o(\lambda) = A \text{ and }  \wt{d}\(\lambda\) = B \big\}
\end{equation}
\begin{theorem}\label{thm:recursion}
For $n>0$
\begin{align*}
f(n; A, B) =& f(n-1; A-1,B-A+1) + f(n-1;A, B-A )  \\ &+ A f(n-1;A,B-A+1  ) + (A+1) f(n-1; A+1,B-A ).
\end{align*}
with initial condition $f(0; A, B) = 0$ for all $(A,B) \ne (0,0)$ and $f(0;0,0) = 1$.
\end{theorem}

Therefore, to find the number of partitions of $[n]$ with dimension exponent equal to $k$, it suffices
to compute $f(n, 0, k)$ for $k$ and $n$.  Figure \ref{fig:dimensionHistograms} gives the
histograms of the dimension exponent when $n=20$ and $n=100$.
With increasing $n$, these distributions tend to normal with mean and variance given in Theorem
\ref{thm:asymptotics}. This approximation is already apparent for $n=20$.
\begin{figure}[h!]
\center
\includegraphics[scale=0.35]{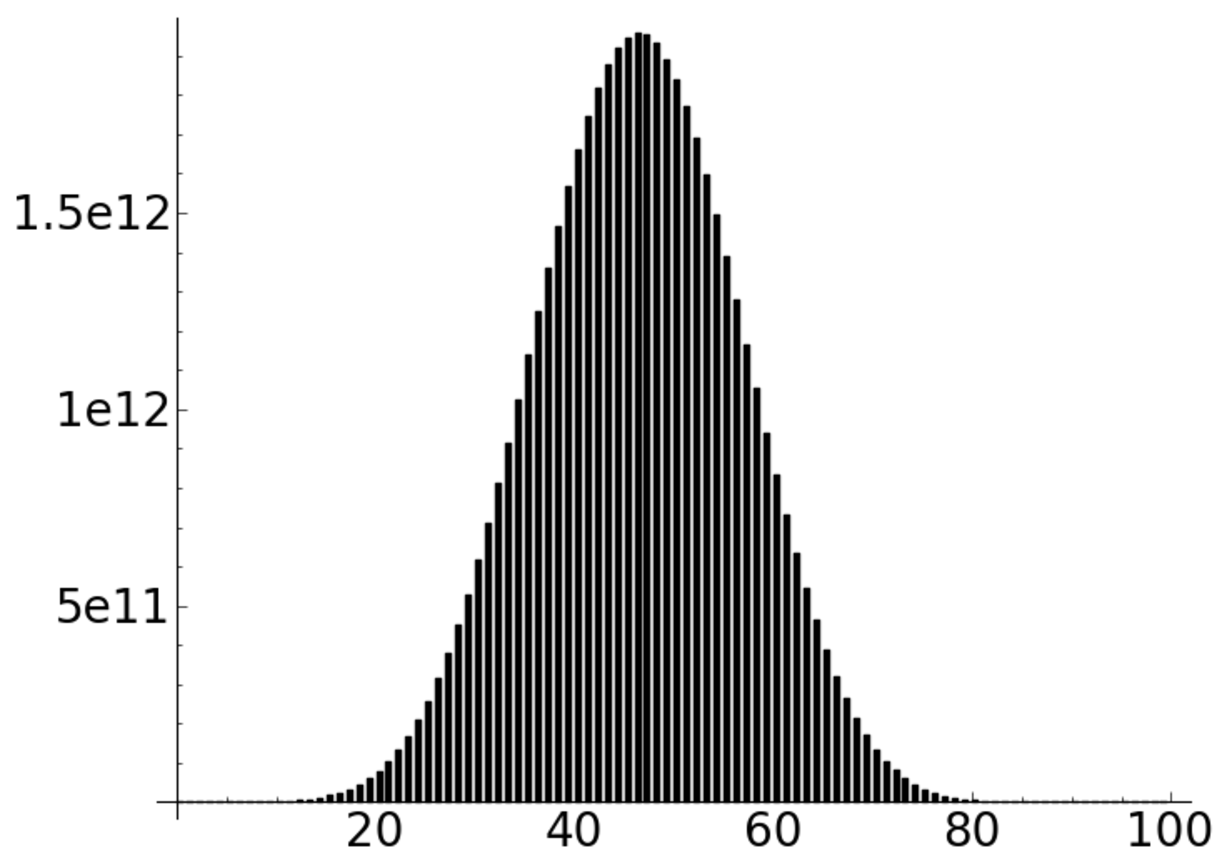}
\includegraphics[scale=0.35]{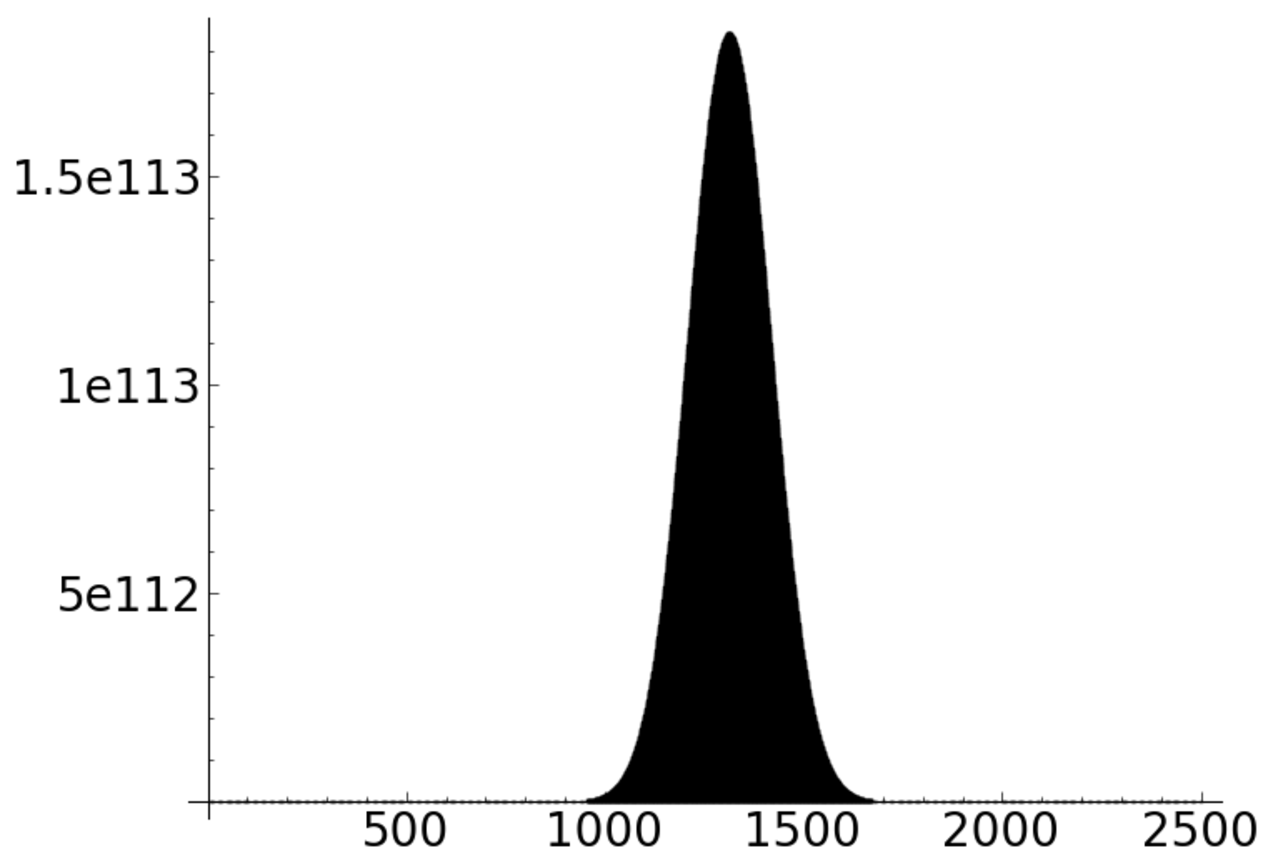}
\caption{Histograms of the dimension exponent counts for $n=20$ and $n=100$.}
\label{fig:dimensionHistograms}
\end{figure}

It is not necessary to compute the entire distribution of the dimension index
to compute the
moment formulae for the dimension exponent.
Namely, it is better to implement the following recursion for the moments.
\begin{corollary}\label{cor:recursion}
Define $M_k(d; n,A):= \sum_{\begin{subarray}{c} \lambda \in \Pars(n) \\ o(\lambda) = A\end{subarray}} d(\lambda)^k$.
Then
\begin{align*}
M_k(d; n,A) =& \sum_{j=0}^k \binom{k}{j} (A-1)^{k-j} M_j(d; n-1, A-1) + \sum_{j=0}^k \binom{k}{j} A^{k-j} M_j(d; n-1,A)\\
&+ A\sum_{j=0}^k \binom{k}{j} (A-1)^{k-j} M_j(d; n-1,A) + (A+1) \sum_{j=0}^{k} \binom{k}{j} A^{k-j} M_j(d; n-1,A+1).
\end{align*}
\end{corollary}
To compute $M(d^k; n)$, then for each $m<n$ this recursion allows us to keep only $k$ values rather than
computing
all $O(m\cdot m^2)$ values of $f(m, A, B)$.   To find the linear relation of Theorem \ref{thm:structure} only  $O(k\cdot k^2)$ values of $M_k(d;n, A)$ are needed.

In analogy, there is a recursion for the intertwining exponent.
Let $f_{(i)}(n,A,B)$ be the number of marked partitions of $[n]$ with intertwining weight equal to $B$ and with $A$ open sets where
the intertwining weight is equal to the number of interlaced pairs
$i \frown j$ and $k \frown \ell$ where $k$ is in a closed set plus the number of triples $i, k, j$ such that
$i\frown j$ and $k$ is in an open set.
\begin{theorem}\label{thm:recursionInter}
With the notation above, the  following recursion holds
\begin{align*}
f_{(i)}(n+1, A, B) =& f_{(i)}(n, A, B) + f_{(i)}(n, A-1, B) \\&+ \sum_{j=0}^A f_{(i)}(n, A+1, B-j) + \sum_{j=0}^{A-1} f_{(i)}(n, A, B-j).
\end{align*}
\end{theorem}
This recursion allows the distribution to be computed rapidly.
Figure \ref{fig:intertwiningHistograms} gives the histograms of the intertwining exponent when $n=20$ and $n=100$.  Again, for increasing $n$ the distribution tends to normal with mean and variance from Theorem
\ref{thm:asymptoticsInter}. The skewness is apparent for $n=20$.
\begin{figure}[h!]
\center
\includegraphics[scale=0.35]{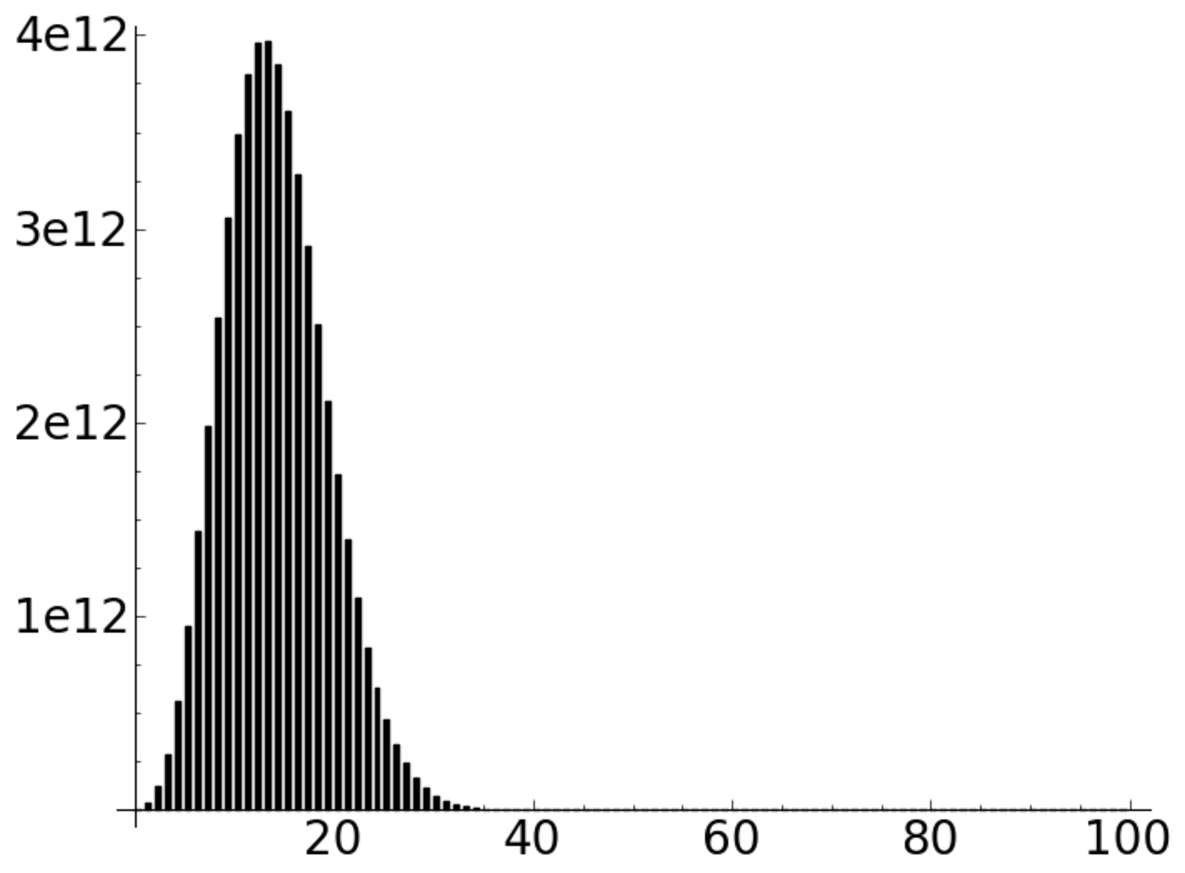}
\includegraphics[scale=0.35]{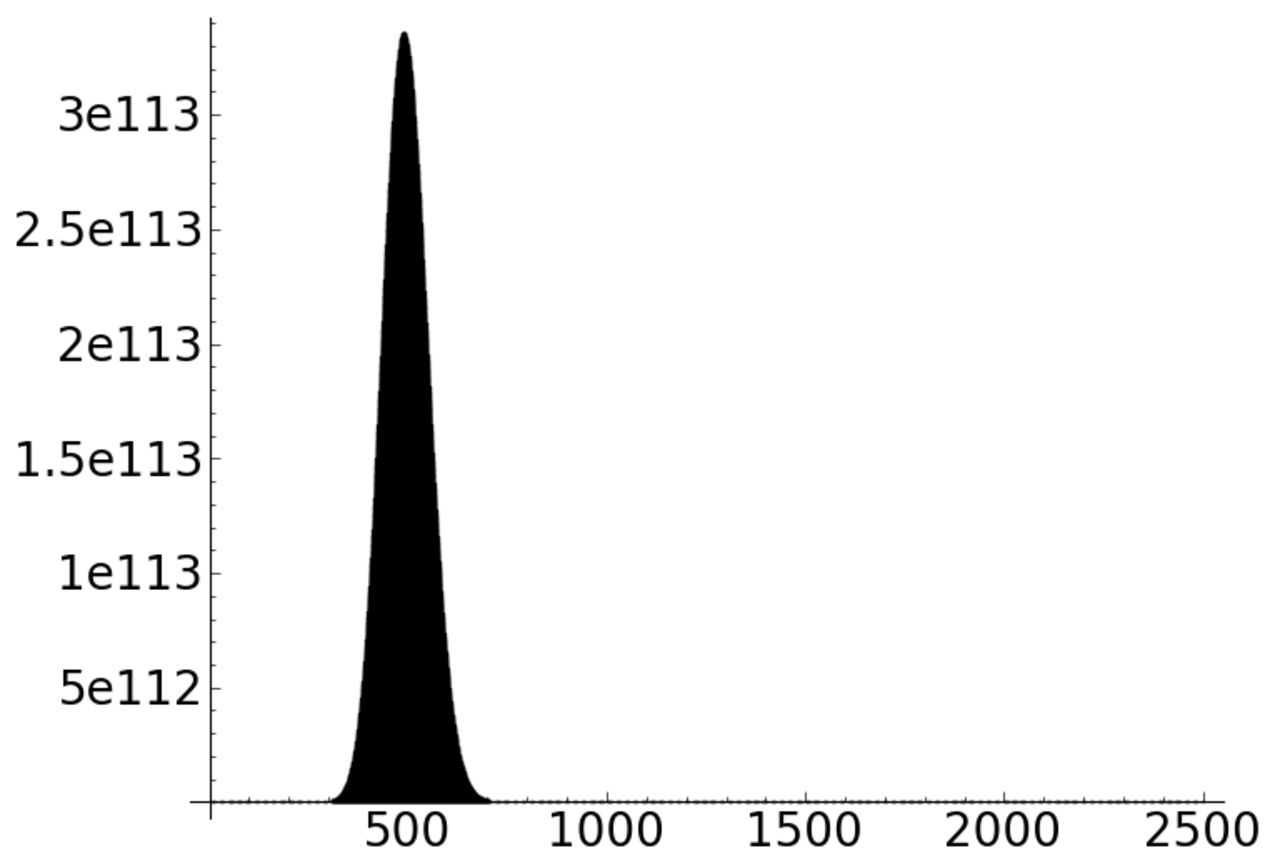}
\caption{Histograms of the dimension exponent counts for $n=20$ and $n=100$.}
\label{fig:intertwiningHistograms}
\end{figure}

\section{Proofs of Recursions, Asymptotics, and Theorem \ref{thm:structure} }\label{sec:Proofs2}
This section gives the proofs of the recursive formulae discussed in Theorems \ref{thm:recursion} and
\ref{thm:recursionInter}. Additionally, this section gives a proof of Theorem \ref{thm:structure} using the three variable
generating function for $f(n, A, B)$.   Finally, it gives
an asymptotic expansion for $B_{n+k}/B_n$ with $k$ fixed and
$n\to \infty$. This asymptotic is used to deduce Theorems \ref{thm:asymptotics} and \ref{thm:asymptoticsInter}.

\subsection{Recursive formulae}\label{sec:Recursion}
This subsection gives the proof of the recursions for $f(n, A, B)$ and $f_{(i)}(n, A, B)$ given
in Theorems \ref{thm:recursion} and \ref{thm:recursionInter}.
The recursion is used in the next subsection to study the generating function
for the dimension
exponent.

\begin{proof}[Proof of Theorem \ref{thm:recursion}]
The four terms of the recursion come from considering the following cases:
(1) $n$ is added to a marked partition of $[n-1]$ as a singleton open set,
(2) $n$ is added to a marked partition of $[n-1]$ as a singleton closed set,
(3) $n$ is added to an open set of a marked partition of $[n-1]$ and that set remains open,
(4) $n$ is added to an open set of a marked partition of $[n-1]$ and that set is closed.
\end{proof}


\begin{proof}[Proof of Theorem \ref{thm:recursionInter}]
The argument is similar to that of Theorem \ref{thm:recursion}.
The same four cases arise.  However, when adding
$n$ to an open set the statistic may increase by any value $j$ and it does so in exactly one way.
\end{proof}

\subsection{The Generating Function for $f(n, A, B)$}\label{sec:Dimension}
This section studies the generating function for $f(n, A, B)$ and deduces
Theorem \ref{thm:structure}. Let
\begin{equation}
F(X,Y,Z):= \sum_{n, A, B \ge0}f(n; A, B) \frac{X^n}{n!} Y^A Z^B
\end{equation}
be the three variable generating function.
Theorem \ref{thm:recursion} implies that
\begin{equation}\label{eqn:diffeq}
\frac{\partial}{\partial X} F(X,Y,Z) = (1+Y) \( F(X, YZ, Z) + F_Y(X, YZ, Z)\),
\end{equation}
where $F_Y$ denotes $\frac{\partial}{\partial Y} F$.

Then $F(X,0,Z)$ is the generating function for the distribution of $d(\lambda)$, i.e.
$$F(X, 0, Z) = \sum_{n =0}^\infty \sum_{\lambda \in \Pars(n)} Z^{d(\lambda)} \frac{X^n}{n!}.$$
Thus, the generating function for the $k$th moment is
$$\sum_{n\ge 0} M(d^k;n) \frac{X^n}{n!} = \(Z \frac{\partial}{\partial Z}\)^k F(X, Y, Z) \big|_{Z = 1, Y = 0}.$$
Consider
\begin{equation}
F_k(X,Y):= \(Z \frac{\partial}{\partial Z}\)^k F(X, Y, Z) \big |_{Z = 1}.
\end{equation}
So $F_k(X,0) = \sum M(d^k;n) \frac{X^n}{n!}$.
\begin{lemma}\label{lem:GF_lem1}
In the notation above,
$$\( \frac{\partial}{\partial X} - (1+Y) \frac{\partial}{\partial Y}\) F_n(X,Y) =
(1+Y) \sum_{k>0} \binom{n}{k} \( \( Y\frac{\partial}{\partial Y}\)^k \(1+ \frac{\partial}{\partial Y}\)
   F_{n-k}(X,Y)\).$$
\end{lemma}
\begin{proof}
From \eqref{eqn:diffeq},
$$\frac{\partial}{\partial X} F_n(X,Y) = (1+Y) \sum_{k} \binom{n}{k} \( \( Y\frac{\partial}{\partial Y}\)^k \(1+ \frac{\partial}{\partial Y}\)
   F_{n-k}(X,Y)\) $$
Hence solving for $F_n$ gives
\begin{equation}\label{eqn:Frecursion}
\( \frac{\partial}{\partial X} - (1+Y) \frac{\partial}{\partial Y}\) F_n(X,Y) =
(1+Y) \sum_{k>0} \binom{n}{k} \( \( Y\frac{\partial}{\partial Y}\)^k \(1+ \frac{\partial}{\partial Y}\)
   F_{n-k}(X,Y)\).
   \end{equation}
\end{proof}

Throughout the remainder $Y = e^{\alpha} - 1$.
Abusing notation, let \begin{equation*}
G_k(X, \alpha) :=
 G_k(X,Y) := F_k(X,Y) \exp\(-(1+Y) (e^X-1)\).
\end{equation*}
The following lemma gives an expression for $G_k(X, \alpha)$ in terms of a
differential operators.
Define the operators
\begin{align*}
R: =& \frac{\partial}{\partial X} - \frac{\partial}{\partial \alpha}\\
S:= & e^{\alpha}\\
T:= & \frac{\partial}{\partial \alpha} + e^{X+\alpha}.
\end{align*}
\begin{lemma}\label{lem:GF_lem2}
Clearly $G_0(X,Y) = 1$.  Moreover,
$$G_{k}(X, \alpha) = \sum_{a, b, c} C^k_{a,b,c}S^a T^b X^c 1,$$
\end{lemma}
\begin{proof}
\eqref{eqn:Frecursion} is equivalent to
\begin{align*}
&\( \frac{\partial}{\partial X} + (1+Y) e^X - (1+Y) \( \frac{\partial}{\partial Y} + e^X\)\) G_n(X,Y)  \\
=& (1+Y) \sum_{k>0} \binom{n}{k} \( Y \( \frac{\partial }{\partial Y} +e^X -1\)\)^k  \( \frac{\partial}{\partial Y} + e^X\) G_{n-k}
\end{align*} Now
$$
\(\frac{\partial}{\partial X} - \frac{\partial}{\partial \alpha}\) G_k(X,\alpha) =
\sum_{\ell > 0} \binom{k}{\ell} \((1-e^{-\alpha}) \( \frac{\partial}{\partial \alpha} + e^{X+\alpha} - e^{\alpha} \)\)^k
\( \frac{\partial}{\partial \alpha} + e^{X+\alpha}\) G_{k-\ell}(X, \alpha)$$
where  a $e^\alpha$ has been commuted through.
Then
\begin{equation}
R G_{k}(X,\alpha) =  \sum_{\ell >0} \binom{k}{\ell}  \( T-TS^{-1} - S\)^\ell T G_{k-\ell}.
\end{equation}
Since $G_k(0,\alpha)=0$ for $k>0$,
\begin{equation}\label{GrecurseEqn}
G_k(X,\alpha) = \int_{0}^X \sum_{\ell >0} \binom{k}{\ell}  \( T-TS^{-1} - S\)^\ell T G_{k-\ell}(t,X+\alpha-t)dt.
\end{equation}
From this  $$G_{k}(X, \alpha) = \sum_{a, b, c} C^k_{a,b,c}S^a T^b X^c 1,$$
for some constants $C^k_{a,b,c}.$
\end{proof}

The next lemma evaluates the terms in the summation of
Lemma \ref{lem:GF_lem2}, thus yielding a generating function for $G_k(X,Y)$ which
resembles that for the Bell numbers.
\begin{lemma}\label{lem:GF_lem3}
$$ \( T^\ell 1\)\big |_{\alpha =0} \exp\( e^X - 1\) = \sum_{n\ge 0} B_{n+\ell} \frac{X^n}{n!}.$$
\end{lemma}
\begin{proof}
It is easy to see by induction on $\ell$ that $T^\ell 1$ is a polynomial in $e^{X+\alpha}$.
Thus $$T^\ell 1 = \left( \frac{\partial}{\partial X} + e^{X+\alpha}\right)^\ell 1.$$ Hence
$$T^\ell 1\big|_{\alpha=0} = \left( \frac{\partial}{\partial X} + e^{X}\right)^\ell 1.$$ From this, it is easy to see that
$$T^\ell 1\big|_{\alpha=0}\exp\( e^X - 1\) = \frac{\partial^\ell}{\partial X^\ell}\exp\( e^X - 1\).$$ And the result follows.
\end{proof}

Lemmas \ref{lem:GF_lem2} and \ref{lem:GF_lem3} readily yield
the following expression for the moments of the dimension exponent as a shifted
Bell polynomial.
\begin{lemma} For each $k\ge 0$ and $n\ge 0$
$$
M(d^k;n) = \sum_{a,b,c} C^k_{a,b,c} n(n-1)\cdots (n-c+1)B_{n+b-c}.
$$
\end{lemma}

Theorem \ref{thm:structure} needs  some further constraints on the degrees of terms in this polynomial.
The following lemma yields the claimed bounds for the degrees.
\begin{lemma}
In the notation above, $C^k_{a,b,c}=0$ unless all of the following hold:
\begin{enumerate}
\item $c \leq b$.
\item $c < b$ unless $a=0$.
\item $b\leq 2k$.
\item $3c-b \leq k.$
\item $3c-b\leq k-2$ if $a\neq 0$.
\end{enumerate}
\end{lemma}
\begin{proof}
Let $H_{a,b,c}(X,\alpha) = S^a T^b X^c 1.$ Using Equation \eqref{GrecurseEqn},
write $C^k_{a,b,c}$ in terms of the $C^\ell_{a,b,c}$ for $\ell<k$. To do this requires understanding
$$
\int_0^X H_{a,b,c}(t,X+\alpha-t)dt.
$$
As a first claim:
 if $a=0$, then the above is simply $\frac{1}{c+1}H_{0,b,c+1}$. This is seen easily from the fact that $R$ commutes with $T$. For $a\neq 0$, it is easy to see that this is a linear combination of the $H_{a,b,c'}$ over $c'\leq c$, and of $H_{0,b',0}$ over $b'\leq b$.

The desired properties can now be proved by
induction on $k$. It is clear that they all hold for $k=0$. For larger $k$,
assume that they hold for all $k-\ell$, and use Equation \ref{GrecurseEqn} to prove them for $k$.

By the inductive hypothesis, the $TG_{k-\ell}$ are linear combinations of $H_{a,b,c}$ with $c< b$. Thus $\( T-TS^{-1} - S\)^\ell T G_{k-\ell}$ is a linear combination of $H_{a,b,c}$'s with $b>c$. Thus, by Equation \eqref{GrecurseEqn}, $G_k$ is a linear combination of $H_{a,b,c}$'s with $c \leq b$ and $a=0$ or with $c<b$. This proves properties 1 and 2.

By the inductive hypothesis the $G_{k-\ell}$ are linear combinations of $H_{a,b,c}$ with $b\leq 2(k-\ell)$. Thus $\( T-TS^{-1} - S \)^\ell T G_{k-\ell}$ is a linear combination of $H_{a,b,c}$'s with $b\leq 2k+1-\ell \leq 2k$. Thus, by Equation \eqref{GrecurseEqn}, $G_k$ is a linear combination of $H_{a,b,c}$'s with $b\leq 2k$. This proves property 3.

Finally, consider the contribution to $G_k$ coming from each of the $G_{k-\ell}$ terms. For $\ell=1$, $G_{k-\ell}$ is a linear combination of $H_{a,b,c}$'s with $3c-b\leq k-3$ if $a\neq 0$, $3c-b\leq k-1$ if $a=0$. Thus $TG_{k-\ell}$ is a linear combination of $H_{a,b,c}$'s with $3c-b\leq k-3$ if $a\neq 0$, and $3c-b\leq k-2$ otherwise. Thus, $\( T-TS^{-1} - S \)^\ell T G_{k-\ell}$ is a linear combination of $H_{a,b,c}$'s with $3c-b\leq k-3$ if $a=0$, and $3c-b\leq k-2$ otherwise. Thus the contribution from these terms to $G_k$ is a linear combination of $H_{a,b,c}$'s with $3c-b\leq k$ and $3c-b\leq k-2$ if $a\neq 0$. For the terms with $\ell>1$, $G_{k-\ell}$ is a linear combination of $H_{a,b,c}$'s with $3c-b\leq k-2$ and $3c-b\leq k-4$ when $a\neq 0$. Thus, $TG_{k-\ell}$ is a linear combination of $H_{a,b,c}$'s with $3c-b\leq k-3$, as is $\( T-TS^{-1} - S \)^\ell T G_{k-\ell}$. Thus, the contribution of these terms to $G_k$ is a linear combination of $H_{a,b,c}$'s with $3c-b\leq k$ and $3c-b\leq k-3$ if $a\neq 0$. This proves properties 4 and 5.

This completes the induction and proves the Lemma.
\end{proof}

From this Lemma, it is easy to see that
$$
M(d^k;n) = \sum_{\ell=0}^{2k} B_{n+\ell}P_{k,\ell}(n)
$$
for some polynomials $P_{k,\ell}(n)$ with $\deg(P_{k,\ell}) \leq \min(2k-\ell,k/2+\ell/2)$.

\subsection{Asymptotic Analysis}\label{sec:Asymptotic}
This section presents some asymptotic analysis of the Bell numbers and ratios of Bell numbers.
These results yield Theorems \ref{thm:asymptotics} and \ref{thm:asymptoticsInter}.  Similar analysis
can be found in \cite{knuth4a}.
\begin{proposition}\label{prop:Bn+k}
Let $\alpha_n$ be the solution to
$$ue^u = n+1$$
and let
$$\zeta_{n,k} := e^{\alpha_n} \( 1+ \frac{1}{\alpha_n}\) +\frac{k}{\alpha_n^2} = \frac{(n+1)(\alpha_n+1) + k}{\alpha_n^2}.$$
Then
$$B_{n+k} = \frac{(n+k)!}{\sqrt{2\pi} e} \zeta_{n,k}^{-\frac{1}{2}}  \exp\( e^{\alpha_n} - (n+k + 1) \log(\alpha_n)\) \(1 + O\( e^{-\alpha_n}\)\).$$
More precisely, for $T \ge 0$
\begin{align*}
B_{n+k} =  \frac{(n+k)!}{\sqrt{2\pi} e} &  \zeta_{n,k}^{-\frac{1}{2}}  \exp\( e^{\alpha_n}
- (n+k + 1) \log(\alpha_n)\) \\
& \times \( 1+ \sum_{m=1}^T R_{m,k}(\alpha_n) \frac{1}{n^{m}}
 + O\( \( \frac{\alpha_n}{n} \)^{T+1} \) \).
 \end{align*}
where $R_{m,k}$ are rational functions.  In particular
{\small
\begin{align*}
R_{1, k}(u) = &\frac{\(	
(-12k^2 + 24k - 2) + (-24k^2 + 24k + 18)u + (-12k^2 - 12k + 20)u^2 + (-12k + 3)u^3 -2u^4 \)}{24(u+1)^3}   \\
R_{2,k}(u) = & \frac{(144k^4 - 384k^3 + 624k^2 - 1152k + 100) + (576k^4 - 576k^3 + 816k^2 - 3264k - 648)u}{1152(u+1)^6}   \\
&  \hspace{.5in} + \frac{(864k^4 + 1056k^3 + 432k^2 - 6384k - 1292)u^2}{1152(u+1)^6}  \\
& \hspace{.5in}   + \frac{(576k^4 + 2784k^3 + 2280k^2 - 7440k - 2604)u^3 }{
         1152(u+1)^6}  \\
&\hspace{.5in}  + \frac{(144k^4 + 2016k^3 + 3888k^2 - 3552k - 2988)u^4 +
(480k^3 + 2328k^2 + 72k - 1800)u^5 }{1152(u+1)^6}   \\
&\hspace{.5in}  + \frac{(480k^2 + 600k - 551)u^6 + (144k - 60)u^7 + 4u^8 }{1152(u+1)^6}
\end{align*}
}
\end{proposition}
\begin{proof}
The proof is very similar to the traditional saddle-point method for approximating $B_{n}$.
The idea is to evaluate at the saddle point for $B_n$ rather than for $B_{n+k}$.
We follow the proof in Chapter 6 of \cite{deBruijn}.

By Cauchy's formula,
$$\frac{2\pi i e}{(n+k)!} B_{n+k} = \int_C \exp(e^z) z^{-n-k-1} dz$$
where $C$ encircles the origin once in the positive direction.
Deform the path to a vertical line
$u-i\infty$ to $u+i\infty$ by taking a large segment of this line and a large semi-circle going around the origin.  As the
radius, say $R$, is taken to infinity the factor $z^{-n-k-1} = O(R^{-n-k-1})$ and $\exp(e^z)$ is bounded in the half-plane.

Choose $u = \alpha_n$ and then
$$\frac{2\pi e}{(n+k)!} B_{n+k}
   = \exp\( e^{\alpha_n} - (n+k+1) \log(\alpha_n)\) \int_{-\infty}^\infty \exp\( \psi_{n,k}(y)\) dy$$
where
$$\psi_{n,k}(y) = e^{\alpha_n} \( (e^{iy}-1) - \frac{n+1+k}{e^{\alpha_n}} \log\( 1+ i y \alpha_n^{-1}\)\).$$
The the real part has maxima around $y  = 2\pi m$ for each integer $m$, but using
$\log\( 1+ y^2 \alpha_n^{-2}\) > \frac{1}{2} y^{2} \alpha_n^{-2}$ for $\pi < y < \alpha_n$
and  $1+ y^2 \alpha_n^{-2} > 2 y \alpha_n^{-1}$ for $y >\alpha_n$
 as in \cite{deBruijn} gives
$$\int_{-\infty}^\infty \exp\(\psi_{n,k}(y)\) dy  = \int_{-\pi}^{\pi} \exp\( \psi_{n,k}(y)\) dy
+ O\( \exp\(  - \frac{e^{\alpha_n}}{\alpha_n}\)\).$$
Next, note that
\begin{align*}
\psi_{n,k}(y)
=& -\frac{iky}{\alpha_n} - \( 1+  \frac{n+1+k}{(n+1) \alpha_n}\) \frac{n+1}{\alpha_n} \frac{y^2}{2}
  +  \sum_{m>2}  \( \frac{1}{m!} + (-1)^m \frac{n+1+k}{m\alpha_n^{m-1} (n+1) }\) \frac{n+1}{\alpha_n}  (iy)^m
\end{align*}
where
$\frac{n+1+k}{e^{\alpha_n}} = \alpha_n + k e^{-\alpha_n}$ and $e^{\alpha_n} = \frac{n+1}{\alpha_n}$
were used.
Hence,
\begin{align*}
\psi_{n,k}\(\frac{y}{\sqrt{\zeta_{n,k}}} \)  =& -\frac{ik}{\alpha_n\sqrt{\zeta_{n,k}} } - \frac{y^2}{2}
  +  \sum_{m>2}  \( \frac{1}{m!} + (-1)^m \frac{n+1+k}{m \alpha_n^{m-1} (n+1) }\) \frac{n+1}{\alpha_n}  \(\frac{iy}{\sqrt{\zeta_{n,k} }}\)^m
\end{align*}

Making the change of variables and extending the sum of interval of integration gives
\begin{align*}
\int_{-\infty}^{\infty}& \exp\( \psi_{n,k}(y)\) dy   + O\( \exp\( - \frac{e^{\alpha_n}}{\alpha_n}\)\) \\
&= \int_{-\infty}^\infty e^{-\frac{y^2}{2}} \exp\(  -\frac{ik}{\alpha_n\sqrt{\zeta_{n,k}} }
  +  \sum_{m>2}  \( \frac{1}{m!} + (-1)^m \frac{n+1+k}{\alpha_n^{m-1} (n+1) }\) \frac{n+1}{\alpha_n}
   \(\frac{iy}{\sqrt{\zeta_{n,k} }}\)^m  \) dy .
\end{align*}
Hence,
Taylor expanding around $y = 0$ and
using $$\int_{\R} y^{k} e^{-\frac{y^2}{2}} dy = \begin{cases} 0 & k \equiv 1\pmod{2} \\
\sqrt{2\pi} \frac{k!}{2^{\frac{k}{2}} \( \frac{k}{2}\)!} & k \equiv 0 \pmod{2} \end{cases}$$
gives the desired result.
For more details see \cite{deBruijn}.
\end{proof}

Proposition \ref{prop:Bn+k} yields
\begin{equation}\label{eqn:Bell_asymptotic}
\frac{B_{n+k}}{B_{n}} = \frac{(n+k)!}{n!} \alpha_n^{-k}
\( 1 - \frac{k\alpha_n}{(n+1)(\alpha_n+1)}\)^{-\frac{1}{2}} \(1 + O\( e^{-\alpha_n}\)\).
\end{equation}
Direct application of this result gives the results in Theorems
\ref{thm:asymptotics} and \ref{thm:asymptoticsInter}.

\section{Proofs of Theorems \ref{thm:algebraAbstract} and \ref{thm:structureAbstract} }\label{sec:Proofs1}
This section gives the proofs of Theorems \ref{thm:structureAbstract} and \ref{thm:algebraAbstract}.
This result
implies Theorem \ref{thm:structureInter}. A pair of lemmas which will
be useful in the proof of Theorem \ref{thm:structureAbstract}:

\begin{lemma}\label{lem:bell_sums}
For $B_n$ the Bell numbers, define
$$
g_{r,d,k,s}(n) := n^d \sum_{i=0}^{n-k} \binom{n-k}{i} B_{i+s}r^{n-k-i}
$$
where $r,d,k,s$ are non-negative integers.
Then $g_{r, d, k, s}(n)$ is a shifted Bell polynomial of lower shift index $-k$ and upper shift index
$r+s -k$.
\end{lemma}

\begin{proof}
It clearly suffices to prove that $g_{r,0,k,s}(n)$ is a shifted Bell polynomial.
Since $g_{r, 0, 0, s}(n-k) = g_{r, 0, k,s}(n)$,
it suffices to prove that $g_{r,s}(n):=g_{r,0,0,s}(n)$ is a shifted Bell polynomial.

For this consider the exponential generating function
\begin{align*}
\sum_{n=0}^\infty g_{r,s}(n)\frac{x^n}{n!} = & \sum_{n=0}^\infty \sum_{i=0}^n \binom{n}{i} B_{i+s} r^{n-i} \frac{x^n}{n!} =
\sum_{a=0}^\infty \sum_{b=0}^\infty B_{a+s} r^b \frac{x^{a+b}}{a!b!} \\
 =& \frac{\partial^s}{\partial x^s}\left(\sum_{n=0}^\infty B_n \frac{x^n}{n!} \right)e^{rx} = e^{rx}
 \frac{\partial^s}{\partial x^s}\left(e^{e^x-1}\right).
\end{align*}
This is easily seen to be equal to $e^{e^x-1}$ times a polynomial in $e^x$.
On the other hand,  the exponential generating function for $g_{0,s}(n)=B_{n+s}$
(with $S(s,a)$ the Stirling number of the second kind) is
$$
\frac{\partial^s}{\partial x^s}\left(e^{e^x-1}\right) = e^{e^x -1} \sum_{a = 0}^s S(s, a) e^{ax} ,
$$
which is $e^{e^x-1}$ times a polynomial in $e^x$ of degree exactly $s$.
From this, conclude that the space of all polynomials in $e^x$ times $e^{e^x-1}$ is spanned by the set of generating functions for the sequences $B_{n+s}$ as $s$ runs over non-negative integers.
In particular, $e^{qx} e^{e^x - 1} = \sum_{n=0}^\infty \sum_{a=0}^q \beta_{q, a} B_{n+a} \frac{x^n}{n!}$
for some rational numbers $\beta_{q, a}$.
Since the generating function for $g_{r,s}(n)$ lies in this span.  Moreover,
$$g_{r, 0, k, s}(n) = \sum_{b=-k}^{r+s-k} \alpha_{r+s, b+k} B_{n+b}$$
for some rational numbers $\alpha_{n,m}$.  This gives the result.
\end{proof}

For a sequence, ${\bf r} = \{r_0, r_1, \cdots, r_k\}$, of rational numbers  and a
polynomial $Q \in \Q[y_1, \cdots, y_k, m]$ define
\begin{equation}
M(k,Q, {\bf r},n,x):=
\sum_{1\leq x_1<x_2<\ldots<x_k\leq n} Q(x_1,\ldots,x_k,n)\prod_{i=0}^k(x+r_i)^{x_{i+1}-x_i-1},
\end{equation}
where $x_0=0,x_{k+1}=n+1.$

\begin{lemma}\label{lem:gf_sum}
Fix $k$, let $Q \in \Z[y_1, \cdots, y_k, m]$ and  ${\bf r} = \{r_0,  r_1, \cdots, r_k\}$
be a sequence of rational numbers.
As defined above,  $M(k,Q,{\bf r},n,x)$ is a rational linear combination of terms of the form
$$
F(n)G(x)(x+r_i)^{n-k},
$$
where $F\in\Q[n],G\in\Q[x]$ are polynomials.
\end{lemma}
\begin{proof}
The proof is by induction on $k$. If $k=0$ then definitionally, $M(k,Q,{\bf r},n,x) = Q(n)(x+r_0)^n$,
providing a base case for our result.   Assume that the
lemma holds for $k$ one smaller. For this,  fix the values of $x_1,\ldots,x_{k-1}$ in the
sum and consider the resulting sum over $x_k$. Then
\begin{align*}
M(k,Q,{\bf r},n,x)= & \sum_{1\leq x_1<x_2<\ldots<x_{k-1}\leq n-1} \prod_{i=0}^{k-2}(x+r_i)^{x_{i+1}-x_i-1} \\ & \hspace{.6in} \times  \sum_{x_{k-1}<x_k\leq n} Q(x_1,\ldots,x_k,n)(x+r_{k-1})^{x_k-x_{k-1}-1}(x+r_k)^{n-x_k}.
\end{align*}

\noindent
Consider the inner sum over $x_k$:

If $r_{k-1}=r_k$, then the product of the last two terms is always $(x+r_k)^{n-x_{k-1}-2}$, and thus the sum is some polynomial in $x_1,\ldots,x_{k-1},n$ times $(x+r_k)^{n-x_{k-1}-2}$. The remaining sum over $x_1,\ldots,x_{k-1}$ is exactly of the form $M(k-1,Q',{\bf r'},n-1,x)$, for some polynomial $Q'$, and thus, by the inductive hypothesis, of the correct form.

If $r_{k-1}\neq r_k$ the sum is over pairs of non-negative integers $a={x_k-x_{k-1}-1}$ and $b={n-x_k-1}$ summing to $n-x_{k-1}-2$ of some polynomial, $Q'$ in $a$ and $n$ and the other $x_i$ times $(x+r_{k-1})^a(x+r_k)^b$. Letting $y=(x+r_{k-1})$ and $z=(x+r_k)$, this is a sum of $Q'(x_i,n,a)y^az^b$. Let $d$ be the $a$-degree of $Q'$. Multiplying this sum by $(y-z)^{d+1}$, yields, by standard results, a polynomial in $y$ and $z$ of degree $n-x_{k-1}-2+(d+1)$ in which all terms have either $y$-exponent or $z$-exponent at least $n-x_{k-1}-1$. Thus this inner sum over $x_k$ when multiplied by the non-zero constant $(r_{k-1}-r_k)^{d+1}$ yields the sum of a polynomial in $x,n,x_1,\ldots,x_{k-1}$ times $(x+r_{k-1})^{n-x_{k-1}-2}$ plus another such polynomial times $(x+r_{k-1})^{n-x_{k-1}-2}$. Thus, $M(k,Q,{\bf r},n,x)$ can be written as a linear combination of terms of the form $G(x)M(k-1,Q',{\bf r'},n,x)$. The inductive hypothesis is now enough to complete the proof.
\end{proof}

Turn next to the proof of Theorem \ref{thm:structureAbstract}.
\begin{proof}[Proof of Theorem \ref{thm:structureAbstract}]
It suffices to prove this Theorem for simple statistics. Thus, it suffices to prove that for any pattern $P$ and polynomial $Q$ that
$$
M(f_{P,Q};n) = \sum_{\lambda\in\Pi_n} f_{P,Q}(\lambda) = \sum_{\lambda\in\Pi_n} \sum_{s\in_P \lambda} Q(s)
$$
is given by a shifted Bell polynomial in $n$.
As a first step,  interchange the order of summation over $s$ and $\lambda$ above. Hence
$$
M(f_{P,Q};n) = \sum_{s\in [n]^k} Q(s) \sum_{\substack{ \lambda\in\Pars(n) \\ s\in_P\lambda}} 1.
$$

To deal with the sum over $\lambda$ above, first consider only the blocks of
$\lambda$ that contain some element of $s$.
Equivalently, let $\lambda'$ be obtained from $\lambda$ by replacing all of the blocks of
$\lambda$ that are disjoint from $s$ by their union.
To clarify this notation,
 let $\Pars'(n)$ denote the set of all set partitions of $[n]$ with at most 1 marked block.
For $\lambda'\in\Pars'(n)$  say that $s\in_P \lambda'$ if $s$ in an occurrence of $P$ in
$\lambda'$ as a regular set partition so that additionally the non-marked blocks of $\lambda'$
are exactly the blocks of $\lambda'$ that contain some element of $s$.
For $\lambda'\in \Pars'(n)$ and $\lambda\in\Pars(n)$,
 say that $\lambda$ is a \emph{refinement} of $\lambda'$ if the unmarked blocks in
$\lambda'$ are all parts in $\lambda$,
or equivalently, if $\lambda$ can be obtained from $\lambda'$ by further partitioning the marked block.
 Denote $\lambda$ being a refinement of $\lambda'$ as $\lambda \vdash \lambda'$.
 Thus, in the above computation of $M(f_{P,Q};n)$, letting $\lambda'$ be the marked partition obtained by replacing the blocks in $\lambda$ disjoint from $s$ by their union:
$$
M(f_{P,Q};n) = \sum_{s\in[n]^k} Q(s) \sum_{\substack{ \lambda'\in\Pars'(n) \\ s\in_P\lambda'}}\sum_{\substack{\lambda\in\Pars(n) \\ \lambda\vdash\lambda'}}1.
$$
Note that the $\lambda$ in the final sum above correspond exactly
to the set partitions of the marked block of $\lambda'$.
For $\lambda'\in\Pars'(n)$, let $|\lambda'|$ be the size of the marked block of
$\lambda'$. Thus,
$$
M(f_{P,Q};n)
  = \sum_{s\in[n]^k} Q(s) \sum_{\substack{ \lambda'\in\Pars'(n) \\ s\in_P\lambda'}} B_{|\lambda'|}.
$$
\begin{remark}
This is valid even when the marked block is empty.
\end{remark}
Dealing directly with the Bell numbers above will prove challenging, so instead
compute the generating function
$$
M(P,Q,n,x) := \sum_{s\in[n]^k} Q(s) \sum_{\substack{ \lambda'\in\Pars'(n) \\ s\in_P\lambda'}}x^{|\lambda'|}.
$$
After computing this,
extract the coefficients of $M(P,Q,n,x)$ and multiply them by the appropriate Bell numbers.

To compute $M(P,Q,n,x)$, begin by computing the value of the inner sum in terms of $s=(x_1<x_2<\ldots<x_k)$ that preserve the consecutivity relations of $P$ (namely those in $\bC(P)$).
Denote the equivalence classes in $P$ by $1,2,\ldots,\ell$.
Let $z_i$ be a representative of this $i^{th}$ equivalence class.
Then an element $\lambda' \in\Pars'(n)$ so that $s\in_P \lambda'$ can be thought of as a
set partition of $[n]$ into labeled equivalence classes $0,1,\ldots,\ell$,
where the $0^{th}$ class is the marked block, and the $i^{th}$ class is the block containing $x_{z_i}$.
Thus  think of the set of such $\lambda'$
as the set of maps $g:[n]\rightarrow \{0,1,\ldots,\ell\}$ so that:
\begin{enumerate} 
\item $g(x_j)=i$ if $j$ is in the $i^{th}$ equivalence class
\item $g(x)\neq i$ if $x<x_j$, $j\in\bF(P)$ and $j$ is in the $i^{th}$ equivalence class
\item $g(x)\neq i$ if $x>x_j$, $j\in\bL(P)$ and $j$ is in the $i^{th}$ equivalence class
\item $g(x)\neq i$ if $x_{j}<x<x_{j'}$, $(j,j')\in\bA(P)$ and $j,j'$ are in the $i^{th}$ equivalence class
\end{enumerate}
It is possible that no such $g$ will exist if one of the latter three properties must be violated by some $x=x_h$.
If this is the case, this is a property of the pattern $P$, and not the occurrence $s$, and thus, $M(f_{P,Q};n)=0$ for all $n$. Otherwise, in order to specify $g$,  assign the given values to $g(x_i)$ and each other $g(x)$ may be independently assigned values from the set of possibilities that does not violate any of the other properties. It should be noted that 0 is always in this set, and that furthermore, this set depends only which of the $x_i$ our given $x$ is between. Thus, there are some sets $S_0,S_1,\ldots,S_k\subseteq \{0,1,\ldots,\ell\}$, depending only on $s$, so that $g$ is determined by picking functions $$\{1,\ldots,x_1-1\}\rightarrow S_0,\{x_1+1,\ldots,x_2-1\}\rightarrow S_1,\ldots,\{x_k+1,\ldots,n\}\rightarrow S_k.$$ Thus the sum over such $\lambda'$ of $x^{|\lambda'|}$ is easily seen to be $$(x+r_0)^{x_1-1}(x+r_i)^{x_2-x_1-1}\cdots(x+r_{k-1})^{x_k-x_{k-1}-1}(x+r_k)^{n-x_k},$$ where $r_i=|S_i|-1$ (recall $\abs{S_i} > 0$, because $0 \in S_i$).
For such a sequence, ${\bf r}$ of rational numbers define
\begin{equation}\label{eqn:Mk_big}
M(k,Q,{\bf r},n,x,\bC(P)):=
\sum_{\begin{subarray}{c} 1\leq x_1<x_2<\ldots<x_k\leq n \\ \abs{x_{i} - x_{j}}=1 \text{ for } (i, j) \in \bC(P) \end{subarray}  }
Q(x_1,\ldots,x_k,n)\prod_{i=0}^k(x+r_i)^{x_{i+1}-x_i-1},
\end{equation}
where, as in Lemma \ref{lem:gf_sum},  using the notation $x_0=0,x_{k+1}=n+1.$

Note that the sum is empty if $\bC(P)$ contains nonconsecutive elements. We will henceforth assume that this is not the case. We call $j$ a follower if either $(j-1,j)$ or $(j,j-1)$ are in $\bC(P)$. Clearly the values of all $x_i$ are determined only by those $x_i$ where $i$ is not a follower. Furthermore, $Q$ is a polynomial in these values  and $n$.  If $j$ is the index of the $i$th non-follower then let $y_i = x_j - j+i$.  Now, sequences of $x_i$ satisfying the necessary conditions correspond exactly to those sequences with
$1 \le y_1 < y_2 < \cdots < y_{k - f} \le n-f$ where $f$ is the total number of followers. Thus,
\begin{align*}M(k, Q, {\bf r}, n, x,\bC(P)) & = \sum_{1\leq y_1<y_2<\ldots<y_{k-f}\leq n-f  }
\wt{Q}(y_1,\ldots,y_k,n)\prod_{i=0}^k(x+\wt{r_i})^{y_{i+1}-y_i-1}\\
& = M(k-f,\wt{Q},{\bf \wt{r}},n-f,x).\end{align*}
where the $\wt{r_i}$ are modified versions of the $r_i$ to account for the change from $\{ x_j\}$ to $\{ y_i\}$. In particular, if $x_j$ is the $(i+1)^{st}$ non-follower, then $\wt{r_i} = r_{j-1}$.

By Lemma \ref{lem:gf_sum}, $M(k-f,\wt{Q},{\bf \wt{r}},n-f,x)$ is a linear combination of terms of the form
$F(n) G(x) (x+r_i)^{n-k}$ for polynomials $F\in \Q[n]$ and $G \in \Q[x]$.
Thus, $M(f_{P,Q};n)$ can be written as a linear combination of terms of the form
$g_{r, d, \ell, s}(n)$ where $\ell$ is the number of equivalence classes in $P$
and $r, d, s$ are non-negative integers.
Therefore, by Lemma \ref{lem:bell_sums} $M(f_{P, Q};n)$ is a shifted Bell polynomial.

The bound for the upper shift index follows from the fact that
$M(f_{P,Q};n) = O(n^N B_n)$ and by \eqref{eqn:Bell_asymptotic} each term $n^\alpha B_{n+\beta}$ is of
an asymptotically distinct size.
To complete the proof of the result it is sufficient to bound the lower shift index of the
Bell polynomial.
By \eqref{eqn:Mk_big} it is clear the largest power of $x$ in each term is $(n-k)$.  Thus, from Lemma
\ref{lem:bell_sums},
 the resulting shift Bell polynomials can be written with minimum lower shift index $-k$.
 This completes the proof.
\end{proof}


Next  turn to the proof of Theorem \ref{thm:algebraAbstract}.
To this end,  introduce some notation.
 \begin{definition}
 Given three patterns $P_1,P_2,P_3$, of lengths $k_1,k_2,k_3$,
 say that a \emph{merge} of $P_1$ and $P_2$ onto $P_3$ is a pair of strictly increasing functions $m_1:[k_1]\rightarrow [k_3]$, $m_2:[k_2]\rightarrow [k_3]$ so that
\begin{enumerate}
\item $m_1([k_1])\cup m_2([k_2]) = [k_3]$
\item $m_1(i)\sim_{P_3} m_1(j)$ if and only if $i\sim_{P_1} j$, and $m_2(i)\sim_{P_3} m_2(j)$ if and only if $i\sim_{P_2} j$
\item $i\in\bF(P_3)$ if and only if there exists either a $j\in\bF(P_1)$ so that $i=m_1(j)$ or a $j\in\bF(P_2)$ so that $i=m_2(j)$
\item $i\in\bL(P_3)$ if and only if there exists either a $j\in\bL(P_1)$ so that $i=m_1(j)$ or a $j\in\bL(P_2)$ so that $i=m_2(j)$
\item $(i,i')\in\bA(P_3)$ if and only if there exists either a $(j,j')\in\bA(P_1)$ so that $i=m_1(j)$ and $i'=m_1(j')$ or a $(j,j')\in\bA(P_2)$ so that $i=m_2(j)$ and $i'=m_2(j')$
\item $(i, i') \in \bC(P_3)$ if and only if there exists either a $(j,j')\in\bC(P_1)$ so that $i=m_1(j)$ and $i'=m_1(j')$ or a $(j,j')\in\bC(P_2)$ so that $i=m_2(j)$ and $i'=m_2(j')$
\end{enumerate}
Such a merge is denoted as $m_1,m_2:P_1,P_2\rightarrow P_3$.
\end{definition}
Note that the last four properties above imply that given $P_1$ and $P_2$,  a merge
(including a pattern $P_3$) is uniquely defined by maps $m_1,m_2$ and
an equivalence relation $\sim_{P_3}$ satisfying (1) and (2) above.
\begin{lemma}\label{lem:merge}
Let $P_1$ and $P_2$ be patterns.  For any $\lambda$ there is a one-to-one correspondence:
\begin{equation}\label{eqn:mergepat}
\left\{(s_1, s_2): s_1\in_{P_1} \lambda, s_2\in_{P_2}\lambda \right\}\leftrightarrow
\left\{ P_3,
s_3 \in_{P_3} \lambda, \text{ and } m_1,m_2:P_1,P_2\rightarrow P_3\right\}.
\end{equation}

Moreover,  under this correspondence
\begin{align}\label{eqn:mergepoly}
 Q_{m_1, m_2, Q_1, Q_2}(s_3) :=& Q_1(z_{m_1(1)},z_{m_1(2)},\ldots,z_{m_1(k_1)},n)Q_2(z_{m_2(1)},z_{m_2(2)},\ldots,z_{m_2(k_2)},n)  \nonumber \\
 =Q_1(s_1)Q_2(s_2).
\end{align}
\end{lemma}
\begin{proof}
Begin by demonstrating the bijection defined by Equation \eqref{eqn:mergepat}.
On the one hand, given $s_3\in_{P_3}\lambda$ given by  $z_1<z_2<\ldots<z_{k_3}$
and $m_1,m_2:P_1,P_2\rightarrow P_3$,   define $s_1$ and $s_2$ by the
sequences $z_{m_1(1)}<z_{m_1(2)}<\ldots <z_{m_1(k_1)}$
and $z_{m_2(1)}<z_{m_2(2)}<\ldots <z_{m_2(k_2)}$.
It is easy to verify that these are occurrences of the patterns $P_1$ and $P_2$
and furthermore that equation \eqref{eqn:mergepoly} holds for this mapping.

This mapping has a unique inverse:
Given $s_1$ and $s_2$, note that $s_3$ must equal the union $s_1\cup s_2$.
Furthermore, the maps $m_a$, for $a= 1, 2$, must be given by the unique function
so that $m_a(i)=j$ if and only if the $i^{th}$ smallest element of $s_a$ equals the $j^{th}$
smallest element of $s_3$. Note that the union of these images must be all of $[k_3]$.
In order for $s_3$ to be an occurrence of $P_3$ the equivalence relation $\sim_{P_3}$ must be that $i\sim_{P_3} j$ if and only if the $i^{th}$ and $j^{th}$ elements of $s_3$ are equivalent under $\lambda$.  Note that since $S_1$ and $S_2$ were occurrences of $P_1$ and $P_2$, that this must satisfy condition (2) for a merge. The rest of the data associated to $P_3$ (namely $\bF(P_3),\bL(P_3)$,  $\bA(P_3)$, and $\bC(P_3)$)
is now uniquely determined by $m_1,m_2,P_1,P_2$ and the fact that $P_3$ is a merge of $P_1$ and $P_2$ under these maps. To show that $s_3$ is an occurrence of $P_3$  first note that by construction the equivalence relations induced by $\lambda$ and $P_3$ agree. If $i\in\bF(P_3)$, then there is a $j\in\bF(P_a)$ with $i=m_a(j)$ for some $a,j$. Since $s_a$ is an occurrence of $P_a$, this means that the $j^{th}$ smallest element of $s_a$ in in $\First(\lambda)$. On the other hand, by the construction of $m_a$, this element is exactly $z_{m_a(j)} = z_i$. This if $i\in\bF(P_3)$, $z_i\in\First(\lambda)$. The remaining properties necessary to verify that $S_3$ is an occurrence of $P_3$ follow similarly. Thus, having shown that the above map has a unique inverse,  the proof of the lemma is complete.
\end{proof}

Recall, the number of singleton blocks is denoted $X_1$ and it is a simple statistic.
To illustrate this lemma return to the example of $X_1^2$ discussed prior to the lemma.
Let $P_1 = P_2$ be the pattern of length 1 with $\bA(P_1) = \phi$, $\bF(P_1) =  \bL(P_1) = 1$.
Then there are
 five possible merges of $P_1$ and $P_2$ into some pattern $P_3$.  The first choice of $P_3$ is $P_1$ itself.  In which case
$m_1(1) = m_2(1) = 1$. The latter choices of $P_3$ is the pattern of length 2 with
$\bF(P_3) = \bL(P_3) = \{ 1, 2\}, \bA(P_3)=\emptyset$. The equivalence relation on $P_3$ could be either the trivial one or the one that relates $1$ and $2$ (though in the latter case the pattern $P_3$ will never have any occurrences in any set partition). In either of these cases, there is a merge with
$m_1(1) = 1$ and $m_2(1) = 2$ and a second merge with $m_1(1) = 2$ and $m_2(1) = 1$.
As a result,
\begin{align*}
M(X_1^2; n) =& \sum_{\lambda \in \Pars(n)} X_1(\lambda)^2
= \sum_{\lambda \in \Pars(n)} \( \sum_{\begin{subarray}{c} x_1 \\  x_1 \in \First(\lambda) \\
x_1 \in \Last(\lambda) \end{subarray} } 1 \)^2 \\
=& \sum_{\lambda \in \Pars(n)} \sum_{ \begin{subarray}{c} x_1 \\ x_1 \in \First(\lambda) \\
x_1 \in \Last(\lambda) \end{subarray}}  \sum_{ \begin{subarray}{c} y_1 \\
y_1 \in \First(\lambda) \\
y_1 \in \Last(\lambda) \end{subarray}} 1
\ \ = \ \
2 \sum_{\lambda \in \Pars(n)} \sum_{ \begin{subarray}{c} x_1 <x_2 \\ x_1, x_2 \in \First(\lambda) \\
x_1, x_2 \in \Last(\lambda) \end{subarray}} 1 + \sum_{\lambda \in \Pars(n)} \sum_{\begin{subarray}{c} x_1 \\  x_1 \in \First(\lambda) \\
x_1 \in \Last(\lambda) \end{subarray} } 1
\end{align*}

\begin{proof}[Proof of Theorem \ref{thm:algebraAbstract}]
The fact that statistics are closed under pointwise addition and scaling follows immediately from the definition. Similarly, the desired degree bounds for these operations also follow easily. Thus
only closure and degree bounds for multiplication must be proved.
Since every statistic may be written as a linear combination of simple statistics of no greater degree, and since statistics are closed under linear combination, it suffices to prove this theorem for a product of two simple statistics. Thus let $f_i$ be the simple statistic defined by a pattern $P_i$ of size $k_i$ and a polynomial $Q_i$.
It must be shown that $f_1(\lambda)f_2(\lambda)$ is given by a statistic of degree at most $k_1+k_2+\deg(Q_1)+\deg(Q_2)$.

For any $\lambda$
$$
f_1(\lambda)f_2(\lambda) = \sum_{s_1\in_{P_1} \lambda, s_2\in_{P_2} \lambda} Q_1(s_1)Q_2(s_2).
$$
Simplify this equation using Lemma \ref{lem:merge}, writing this as a sum over occurrences of only a single pattern in $\lambda$.

Applying Lemma \ref{lem:merge},
\begin{align*}
f_1(\lambda)f_2(\lambda) & = \sum_{s_1\in_{P_1} \lambda, s_2\in_{P_2} \lambda} Q_1(s_1)Q_2(s_2)\\
& = \sum_{P_3} \sum_{m_1,m_2:P_1,P_2\rightarrow P_3} \sum_{s_3\in_{P_3} \lambda} Q_{m_1,m_2,Q_1,Q_2}(s_3),
\\
& =  \sum_{m_1,m_2:P_1,P_2\rightarrow P_3} f_{P_3, Q_{m_1,m_2,Q_1,Q_2}}(\lambda).
\end{align*}

Thus, the product of $f_1$ and $f_2$ is a sum of simple characters.
Note that the quantity is a polynomial of $s_3$ which is denoted $Q_{m_1,m_2,Q_1,Q_2}(s_3).$
Finally, each pattern $P_3$ has size at most $k_1+k_2$ and each polynomial $Q_{m_1,m_2,Q_1,Q_2}$ has degree at most $\deg(Q_1)+\deg(Q_2)$. Thus the degree of the product is at most the sum of the degrees.
\end{proof}


\section{More Data}\label{sec:examples}\label{sec:Data}
This section contains some data for the dimension and intertwining exponent statistics.
The moment formulae of Theorem \ref{thm:structure} for $k\le 22$ and
the moment formulae for the intertwining exponent for $k\le 12$ have been computed and are available
at \cite{webpage}.    Moreover, the values
$f(n,0, B)$ for $n\le 238$ and $f_{(i)}(n, 0, B)$ for $n\le 146$ are available.
These sequences can also
be found on Sloane's Online Encyclopedia of integer sequences \cite{sloane}.

The remainder of this section contains a small amount of data and observations regarding
 the distributions $f(n,0,B)$ and $f_{(i)}(n,0,B)$
and
regarding the shifted Bell polynomials of Theorems \ref{thm:structure} and \ref{thm:structureInter}.

\subsection{ Dimension Index}\label{sec:dataDimension}
\begin{center}
\begin{figure}[h!]
\begin{tabular}{c| rrrrrrrrrrrrr}
$n \backslash d$ & 0 & 1 & 2& 3 & 4 & 5 & 6& 7 & 8 & 9& 10 & 11 & 12 \\
\hline
0 & 1 & & & & & & & & & & & & \\
1 & 1 & & & & & & & & & & & & \\
2 & 2 & & & & & & & & & & & & \\
3 & 4 & 1 & & & &  & & & & & & & \\
4 & 8 & 4 & 3 & & & & & & & & & & \\
5 & 16 & 12 & 13 & 9 & 2 & & & & & & & & \\
6 & 32 & 32 & 42 & 42 & 35 & 12 & 8 & & & & & & \\
7 & 64 & 80 & 120 & 145 & 159 & 133 & 86 & 52 & 32 & 6 & & & \\
8 & 128 & 192 & 320 & 440 & 559 & 600 & 591 & 440 & 380 & 248 & 164 & 48 & 30
\end{tabular}
\caption{A table of the dimension exponent $f(n, 0, d)$.}\label{tab:dim}
\end{figure}
\end{center}

A couple of easy observations:
It is clear that $$f(n, 0, 0) = 2^{n-1}.$$  That is the number of set partitions of $[n]$ with dimension exponent $0$ is $2^{n-1}$.
Set partitions of $[n]$ that have dimension exponent 0 must have $n$ appearing in a singleton set or
it must appear in a set with $n-1$, thus the result is obtained by recursion.
Additionally, the number of set partitions of $[n]$ with dimension exponent equal to $1$ is $n2^{n-1}$, that is
$$f(n, 0, 1) = n 2^{n-1}.$$
Curiously,
the numbers $f(n,0,B)$ are smooth (roughly they have many small prime factors), for reasonably sized $B$.
This can be established by using the recursion of Theorem \ref{thm:recursion}.
For example,
 \begin{align*}
f(100,0, 979) =  &
2^{11} \cdot 3^7 \cdot 5^3 \cdot7^2 \cdot11 \cdot797 \cdot 12269 \cdot 12721 \\
   &    \cdot 342966248369 \cdot 2647544517313 \cdot 1641377154765701  \\
   & \cdot 16100683847944858147992523687926541327031916811919 \\
f(100,0, 2079)= &
 2^{27} \cdot 3^{14} \cdot 5^7 \cdot 7^4 \cdot 11^2 \cdot 13^2 \cdot 17 \cdot 19 \cdot 23 \cdot 29 \cdot 31 \cdot 24679914019  \\
 & \cdot 58640283519733 \cdot 194838881932339884007114639638682100019517
\end{align*}
Note that $f(100, 0, 979)$ has 111 digits and $f(100, 0, 2079)$ has about 100 digits.

In the notation of Theorem \ref{thm:structure}
these are some values of the first few moments of $d(\lambda)$.
{\small
\begin{verbatim}
P_{3,0}(n) =  0 +1n
P_{3,1}(n) =  -1/3 +6n +3n^2
P_{3,2}(n) =  +8/3 -45n -12n^2
P_{3,3}(n) =  +18 +51n +12n^2 +1n^3
P_{3,4}(n) =  -131/3 -45n -6n^2
P_{3,5}(n) =  +42 +12n
P_{3,6}(n) =  -8

P_{4,0}(n) =  0 +1n +3n^2
P_{4,1}(n) =  -21/2 -18n -20n^2
P_{4,2}(n) =  -36 +116/3n +72n^2 +6n^3
P_{4,3}(n) =  -5/6 -166/3n -162n^2 -24n^3
P_{4,4}(n) =  -103/3 +312n +150n^2 +16n^3 +1n^4
P_{4,5}(n) =  -81/2 -812/3n -90n^2 -8n^3
P_{4,6}(n) =  +409/3 +168n +24n^2
P_{4,7}(n) =  -104 -32n
P_{4,8}(n) =  +16

P_{5,0}(n) =  0 +1n +10n^2
P_{5,1}(n) =  +1036/15 +50/3n -35n^2 +15n^3
P_{5,2}(n) =  -1373/30 +95/2n -180n^2 -110n^3
P_{5,3}(n) =  +4415/3 -1370/3n +2030/3n^2 +300n^3 +10n^4
P_{5,4}(n) =  +47/2 -605/6n -2350/3n^2 -390n^3 -40n^4
P_{5,5}(n) =  +1049/3 +15n +1380n^2 +330n^3 +20n^4 +1n^5
P_{5,6}(n) =  +4673/30 -2485/2n -2750/3n^2 -150n^3 -10n^4
P_{5,7}(n) =  -95/3 +3005/3n +420n^2 +40n^3
P_{5,8}(n) =  -1010/3 -520n -80n^2
P_{5,9}(n) =  +240 +80n
P_{5,10}(n) =  -32

P_{6,0}(n) =  0 +1n +25n^2 +15n^3
P_{6,1}(n) =  +1655/6 +185/6n -309n^2 -120n^3
P_{6,2}(n) =  -661817/90 -17539/15n +1015n^2 +495n^3 +45n^4
P_{6,3}(n) =  +149203/45 +12779/10n +1935/2n^2 -1770n^3 -340n^4
P_{6,4}(n) =  -1118236/45 +36605/3n -3460n^2 +10420/3n^3 +840n^4 +15n^5
P_{6,5}(n) =  -121658/9 +3887/2n -8385/2n^2 -11450/3n^3 -765n^4 -60n^5
P_{6,6}(n) =  -1547/9 +1133n +3485n^2 +3960n^3 +615n^4 +24n^5 +1n^6
P_{6,7}(n) =  -38697/10 +7573/5n -13695/2n^2 -6940/3n^3 -225n^4 -12n^5
P_{6,8}(n) =  -12653/90 +3410n +3965n^2 +840n^3 +60n^4
P_{6,9}(n) =  +665 -2980n -1560n^2 -160n^3
P_{6,10}(n) =  +2060/3 +1440n +240n^2
P_{6,11}(n) =  -528 -192n
P_{6,12}(n) =  +64
\end{verbatim}
}
These formulae exhibit a number of properties.
Here is a list of some of them.
\begin{enumerate}
\item Using the fact that $B_{n+k} \approx n^k B_n$, each moment $M(d^k;n)$
has a number of terms with asymptotic of size equal to $n^{2k}B_n$, up to powers of $\log(n)$ (or $\alpha_n$).
Call these terms the leading powers of $n$.
The leading `power' of $n$ contribution is equal to
$$(n-2T)^kB_{n+k}$$
where $T$ is the operator given by $TB_{m} = B_{m+1}$.
For example the leading order $n$ contributions for the average is
$$nB_{n+1} - 2 B_{n+2}$$
and the leading order contribution for the second moment is
$$n^2B_{n+2} - 4n B_{n+3} + 4 B_{n+4}.$$
Structure of this sort is necessary because of the asymptotic normality of the dimension exponent (see the forthcoming work
\cite{CDKR2}).  The next remark also concerns this sort of structure.

\item The next order $n$ terms of $M(d^k;n)$ have size roughly $n^{2k-1} B_n$ and have the shape
$$\(\sum_{j\ge 0} C_j (-1)^{j+1} \binom{k}{j} n^{k-j} T^{k+j-1}\)B_n$$
where the constants $C_j$ are
$$C_j = 2^{j-3} (17-j) j.$$

\item The generating function for the polynomials $P_{0,k}(n)$ seems to be
$$\sum_{k\ge 0} P_{0,k}(n) \frac{X^k}{k!} = \exp\( (e^X - 1  - X) n\).$$
We do not have a proof of this observation.
\end{enumerate}

As in the introduction, let
$ S_k(d; n) : = \frac{1}{B_n} \sum_{\lambda \in \cS_n} \( d(\lambda) - \frac{1}{B_n}M(d; n)\)^k.$
From Proposition \ref{prop:Bn+k} and the formulae for $M(d^k;n)$ deduced from
Theorem \ref{thm:structure} and stated in Section \ref{sec:dataDimension}
and using {\rm SAGE}  the asymptotic expansion of the first few
$S_k$ are:
\begin{align*}
S_2(d;n)  =& \frac{\alpha_n^2 - 7\alpha_n + 17}{(\alpha_n+1) \alpha_n^3} n^3   \\ &+
\frac{-8\alpha_n^7 - 29\alpha_n^6 - 136\alpha_n^5 - 207\alpha_n^4 + 69\alpha_n^3 + 407\alpha_n^2 + 116\alpha_n - 80}{2\alpha_n^4 (\alpha_n+1)^4}n^2 + O(n) \\
S_3(d;n)  = &   \frac{6\alpha_n^4 - 83\alpha_n^3 + 435\alpha_n^2 - 732\alpha_n - 881}{3(\alpha_n+1)^3\alpha_n^4} n^4
+ O\(\frac{n^3}{\alpha_n^2}\) \\
%
S_4(d;n) =& 3\( \frac{\alpha_n^2 - 7\alpha_n + 17}{(\alpha_n+1) \alpha_n^3}\)^2 n^6 + O\(\frac{n^5}{\alpha_n^3}\)  \\
S_5(d;n) =& \frac{10}{3} \(\frac{\alpha_n^2 - 7\alpha_n + 17}{(\alpha_n+1)\alpha_n^3} \)
\( \frac{6\alpha_n^4 - 83\alpha_n^3 + 435\alpha_n^2 - 732\alpha_n - 881}{(\alpha_n+1)^3\alpha_n^4}\)n^7
+ O\(\frac{n^6}{\alpha_n^4}\) \\
S_6(d;n) =&  15\( \frac{\alpha_n^2 - 7\alpha_n + 17}{(\alpha_n+1) \alpha_n^3}\)^3 n^9 + O\(\frac{n^8}{ \alpha_n^5}\)
 \end{align*}

\begin{remark}
These asymptotics support the claim that the dimension exponent is normally distributed with
mean asymptotic to $\frac{n^2}{\log(n)}$ and standard deviation
$\sqrt{ \frac{n^3}{\log(n)^2}}$. This result will be established  in forthcoming work \cite{CDKR2}.
\end{remark}

\subsection{Intertwining Index}\label{sec:dataInter}

Table \ref{tab:intr} contains the distribution for of the intertwining exponent for the first few $n$.
\begin{center}
\begin{figure}[h!]
\begin{tabular}{c | rrrrrrrrrrrrr}
$n \backslash B$ &  0 & 1 & 2& 3 & 4 & 5 &6&7& 8 & 9 & 10 & 11 & 12\\
\hline
0 & 1 &  & & &  & & & & & &  & &\\
1 & 1 & & &  & & & & & & & &  &\\
2 & 2 &  & & &  & & & & & &  & & \\
3 & 5 & & & & & & & & & & & & \\
4 & 14 & 1 &  & & &  & & & & &  & & \\
5 & 42 & 9 & 1 & & &  & & & & & & & \\
6 & 132 & 55 & 14 & 2 & &  & &  & & & & & \\
7 & 429 & 286 & 120 & 35 & 6 & 1  & & & && &  & \\
8 & 1430 & 1365 & 819 & 364 & 119 & 35 & 7 & 1 & &  & & & \\
9& 4862 & 6188 & 4900 & 2940 & 1394 & 586 & 203 & 59 & 13 & 2 & & & \\
10 & 16796 & 27132 & 26928 & 20400 & 12576 & 6846 & 3246 & 1358 & 493 & 153 & 38 & 8 & 1
\end{tabular}
\caption{A table of the distribution of the intertwining exponent $f_{(i)}(n, 0, B)$.  Recall $i(\lambda) = cr_2(\lambda)$.}
\label{tab:intr}
\end{figure}
\end{center}

In the notation of Theorem \ref{thm:structureInter}
these are some values of the first few moments of $i(\lambda) = cr_2(\lambda)$.
{\small
\begin{verbatim}
Q_{3,0}(n) =  +19/192 -29/96n +1/16n^2 +1/8n^3
Q_{3,1}(n) =  +331/192 -193/96n -17/16n^2 +3/8n^3
Q_{3,2}(n) =  -25/6 -743/96n -1/4n^2 +3/8n^3
Q_{3,3}(n) =  +775/64 +449/96n -1/16n^2 +1/8n^3
Q_{3,4}(n) =  -451/32 -619/96n -15/16n^2
Q_{3,5}(n) =  +2045/192 +75/32n
Q_{3,6}(n) =  -125/64

Q_{4,0}(n) =  +4387/172800 +103/360n -11/32n^2 0n^3 +1/16n^4
Q_{4,1}(n) =  -3343/10800 +787/144n -7/16n^2 -7/4n^3 +1/4n^4
Q_{4,2}(n) =  -25453/3456 +7777/288n -335/48n^2 -23/8n^3 +3/8n^4
Q_{4,3}(n) =  -16681/8640 -4303/288n -49/8n^2 -9/8n^3 +1/4n^4
Q_{4,4}(n) =  +963509/34560 +23891/480n +6n^2 -5/8n^3 +1/16n^4
Q_{4,5}(n) =  -637751/14400 -8197/288n -53/12n^2 -5/8n^3
Q_{4,6}(n) =  +126773/3456 +1745/96n +75/32n^2
Q_{4,7}(n) =  -3425/192 -125/32n
Q_{4,8}(n) =  +625/256

Q_{5,0}(n) =  -107993/138240 -593/69120n +569/1152n^2 -175/576n^3 -5/192n^4 +1/32n^5
Q_{5,1}(n) =  -79109/27648 -67769/7680n +9859/1152n^2 +995/576n^3 -115/64n^4 +5/32n^5
Q_{5,2}(n) =  +5436923/138240 -1228273/11520n +5925/128n^2 +815/96n^3 -925/192n^4 +5/16n^5
Q_{5,3}(n) =  -29849/512 -92287/6912n +1375/16n^2 -65/32n^3 -415/96n^4 +5/16n^5
Q_{5,4}(n) =  +1825783/27648 -2270759/13824n -6497/576n^2 +865/288n^3 -105/64n^4 +5/32n^5
Q_{5,5}(n) =  -1092827/138240 +9971653/69120n +21119/288n^2 +725/96n^3 -145/192n^4 +1/32n^5
Q_{5,6}(n) =  -14859283/138240 -6747031/34560n -44905/1152n^2 -1145/576n^3 -25/64n^4
Q_{5,7}(n) =  +188749/1536 +656965/6912n +7225/384n^2 +125/64n^3
Q_{5,8}(n) =  -2168275/27648 -61625/1536n -625/128n^2
Q_{5,9}(n) =  +258125/9216 +3125/512n
Q_{5,10}(n) =  -3125/1024
\end{verbatim}
}
We conjecture that $Q_{j}(n) = 0$ for all $j<0$.

The formulae stated above for  $S_k(i;n) := \frac{1}{B_n} \sum_{\lambda \in
\Pars(n)} \( i(\lambda) - M(i;n)/B_n\)^k$
with \eqref{eqn:Bell_asymptotic} give
{\small
\begin{align*}
S_2(i;n) =&\frac{3 \alpha_n^2 - 22\alpha_n + 56}{9\alpha_n^3 (\alpha_n+1)} n^3 \\
& + \frac{-16\alpha_n^7 - 52\alpha_n^6 - 204\alpha_n^5 - 155\alpha_n^4 - 126\alpha_n^3 - 12\alpha_n^2 + 230\alpha_n + 175}{8\alpha_n^4(\alpha_n+1)^4}n^2  + O\( n\) \\
S_3(i;n) =&    \frac{(\alpha_n - 5)(4\alpha_n^3 - 31\alpha_n^2 + 100\alpha_n + 99)}{8 \alpha_n^4 (\alpha_n+1)^3} n^4
+ O\(\frac{n^3}{\alpha_n^3}\) \\
S_4(i;n) =& 3 \(\frac{3\alpha_n^2 - 22\alpha_n + 56)^2}{9 \alpha_n^3(\alpha_n+1)}\)^2 n^6
+ O\( \frac{n^5}{\alpha_n^3}\)
\\
 \\
S_5(i;n) =&  5 \frac{(\alpha_n - 5)(3\alpha_n^2 - 22\alpha_n + 56)(4\alpha_n^3 - 31\alpha_n^2 + 100\alpha_n + 99)}{36 \alpha_n^7 (\alpha_n+1)^4} n^7 + O\(n^6\)
\\
S_6(i;n) =& 15 \(\frac{3\alpha_n^2 - 22\alpha_n + 56)^2}{9 \alpha_n^3(\alpha_n+1)}\)^3 n^9 + O(n^8)
\end{align*}
}


\end{document}